\documentclass[11pt]{amsart}
\usepackage{latexsym,amssymb,amsmath}
\newtheoremstyle{custom}
  {3pt}
  {3pt}
  {\slshape}
  {}
  {\bfseries}
  {.}
  { }
   {}
\theoremstyle{custom}
\newtheorem{theorem}{Theorem}[subsection]

\newtheorem{proposition}[theorem]{Proposition}
\newtheorem{proposition/definition}[theorem]{Proposition/Definition}
\newtheorem{lemma}[theorem]{Lemma}
\newtheorem{corollary}[theorem]{Corollary}
\newtheorem{conjecture}[theorem]{Conjecture}

\newtheorem{prop}[theorem]{Proposition}

\theoremstyle{definition}
\newtheorem{definition}[theorem]{Definition}

\newtheorem{example}[theorem]{Example}

\newtheorem{question}[theorem]{Question}

\theoremstyle{remark}
\newtheorem{remark}[theorem]{Remark}

\def\donote#1{\noindent{\bf #1\ }}

\newtheoremstyle{exercise}
  {3pt}
  {6pt}
  {}
  {}
  {\bfseries}
  {:}
  { }
   {}
\theoremstyle{exercise}
\newtheorem{exercise}[theorem]{Exercise}
\newtheoremstyle{exercises}
  {3pt}
  {6pt}
  {}
  {}
  {\bfseries}
  {:}
  {\newline}
   {}

\theoremstyle{exercise}
\newtheorem{exercises}[theorem]{Exercises}

\input epsf
\def\boxit#1{\vbox{\hrule height1pt\hbox{\vrule width1pt\kern3pt
  \vbox{\kern3pt#1\kern3pt}\kern3pt\vrule width1pt}\hrule height1pt}}

\def\BC{\mathbb C}\def\BN{\mathbb N}

\def\BP{\mathbb P}\def\BG{\mathbb G}
\def\pp#1{\mathbb P^{#1}}

\def\pp#1{{\mathbb P}^{#1}}
\def\tdim{{\rm dim}}

\def\hd{,...,}

\def\inv{{}^{-1}}

\def\cO{{\mathcal O}}
\def\CC{\mathbb C}

\def\11{\mathbf 1}
\def\PP{\mathbb P}

\def\s{\sigma}

\def\d{\delta}

\def\ot{{\mathord{ \otimes } }}
\def\op{{\mathord{\,\oplus }\,}}
\def\otc{{\mathord{\otimes\cdots\otimes}\;}}

\def\ra{{\mathord{\;\rightarrow\;}}}

\def\dim{{\rm dim}\;}
\def\La#1{\Lambda^{#1}}

\def\op{\oplus}

\def\ep{\epsilon}
\def\op{\oplus}

\def\s{\sigma}

\def\BP{\mathbb  P}
\def\BC{\mathbb  C}

\def\pp#1{\mathbb  P^{#1}}

\def\tcodim{\text{codim}}

\def\ep{\epsilon}

\def\hd{, \hdots ,}

\def\inv{{}^{-1}}

\def\La#1{\Lambda^{#1}}

\def\pp#1{\mathbb  P^{#1}}
\def\brank{\underline {\mathbf{R}}}
\def\ur{\brank}\def\uR{\brank}

\def\ra{\rightarrow}

\def\tdim{\operatorname{dim}}

\def\tmax{\operatorname{max}}

\def\ctimes{\times \cdots\times}

\def\be{\begin{equation}}
\def\ene{\end{equation}}

\newcommand{\Spec}{\operatorname{Spec}}
\newcommand{\Hom}{\operatorname{Hom}}

\numberwithin{equation}{section}
\usepackage{color}
\newenvironment{red}{\color{red}}{}
\newcommand{\bred}{\begin{red}}
\newcommand{\ered}{\end{red}}
\usepackage{verbatim}
\def\id{\operatorname{id}}
\def\set#1{\left\{#1\right\}}
\def\fromto#1#2{#1, \dotsc, #2}
\def\setfromto#1#2{\set{\fromto{#1}{#2}}}

\newcommand{\ccH}{{\mathcal{H}}}
\newcommand{\ccI}{{\mathcal{I}}}
\newcommand{\ccO}{{\mathcal{O}}}
\newcommand{\ccQ}{{\mathcal{Q}}}
\newcommand{\ccR}{{\mathcal{R}}}

\def\ccF{\mathcal F}
\def\ccH{\mathcal H}
\def\ccI{\mathcal I}
\def\ccO{\mathcal O}

\newcommand{\JaBu}{Jaros\l{}aw Buczy\'n{}ski}
\newcommand{\shortJaBu}{J.~Buczy\'n{}ski}

\renewcommand{\theenumi}{(\roman{enumi})}
\renewcommand{\labelenumi}{\theenumi}

\newcommand{\Sym}{\operatorname{Sym}}

\newcommand{\reduced}[1]{{#1}_{\operatorname{red}}}

\newcommand{\rpp}[1][]{\ensuremath{rpp_{#1}}}
\newcommand{\brpp}[1][]{\ensuremath{b\rpp[#1]}}

\newcommand{\scheme}{R}

\theoremstyle{plain}
\newcounter{manualnumber}
\newtheorem{InternalConjectureWithManualNumber}[manualnumber]{Conjecture}

\def\srdx{\Sigma_r^d(X)}

\theoremstyle{definition}
\newtheorem{notation}[equation]{Notation}

\title[Equations for secant varieties]{Determinantal equations for secant varieties and the
Eisenbud-Koh-Stillman conjecture}
\author[\shortJaBu{}]{\JaBu{}}
\address{\JaBu{} \\
Institute of Mathematics of the
Polish Academy of Sciences\\
ul.~\'Sniadeckich 8\\
P.O.~Box 21\\
00-956~Warszawa, Poland}
\email{jabu@mimuw.edu.pl}

\author[A.~Ginensky]{Adam Ginensky}
\address{Adam Ginensky\\ WH Trading\\
         125~South Wacker Drive Suite~500\\
         Chicago~IL, 60606, USA}
\email{adam.ginensky@yahoo.com}

\author[J.M.~Landsberg]{J.M.~Landsberg}
\address{J.M.~Landsberg \\ Department of Mathematics\\
Texas A\&M University\\
Mailstop~3368\\
College Station, TX~77843-3368, USA }
\email{jml@math.tamu.edu}

\begin{document}
\begin{abstract}
We address
special cases of a  question of Eisenbud on the ideals of secant
varieties of Veronese re-embeddings
of arbitrary varieties. Eisenbud's question generalizes a
conjecture of Eisenbud, Koh and Stillman  (EKS) for curves.
We prove that set-theoretic equations of small secant
   varieties to a  high degree Veronese re-embedding of a smooth variety
   are determined by equations of the ambient Veronese variety and linear equations.
However this is false for singular varieties,
   and we give explicit counter-examples to the EKS conjecture
   for singular curves.
The techniques we use also allow us to  prove a
gap and uniqueness theorem for symmetric tensor rank.
 We put Eisenbud's
  question  in a more general context about  the  behaviour of border rank under specialisation to a linear subspace,
  and provide an overview of conjectures  coming from signal processing and complexity theory in this context.
  \end{abstract}
 \thanks{\shortJaBu{} supported by Maria Sk\l{}odowska-Curie Outgoing Fellowship ``Contact Manifolds'',
         Landsberg supported by NSF grant   DMS-1006353}
\maketitle

\section{Introduction}

The starting point of this paper was the observation that aspects of conjectures and questions originating in
signal processing, computer science, and algebraic geometry all amounted to assertions regarding
linear sections of secant varieties of Segre and Veronese varieties. In this paper we focus on
linear sections of Veronese varieties  to (i) reduce a question of Eisenbud regarding
arbitrary varieties to the case of projective space, (ii) give explicit counter-examples to
a 20 year old conjecture of Eisenbud, Koh and Stillmann (Conjecture~\ref{EKSconj}), and (iii) prove a
uniqueness theorem for tensor decomposition (Theorem~\ref{uniquenesscora}) that should be useful
for applications to signal processing (more precisely, blind source separation, see, e.g.~\cite{ComoJ10}).

We work over the base field of complex numbers $\CC$.

By a  \emph{variety},  we mean an algebraic integral scheme over
 the   complex numbers.
All our varieties will be  projective,     a reader   outside of algebraic geometry
  may simply think  of a variety as the  zero set  of
   a collection of   homogeneous polynomials in a projective  space.
An interested reader may easily generalise some of our results  to    reduced projective  schemes.

\subsection{Secant varieties of Veronese re-embeddings}

Fix a projective variety $X \subset \PP V$, an integer $r \ge 1$ and choose a sufficiently large $d\in \BN$.
The main objective of this paper is to compare the $r$-th secant variety of $d$-th Veronese embeddings of $X$ and $\PP V$,
  denoted, respectively, $\sigma_r(v_d (X))$ and $\sigma_r(v_d (\PP V))$.
Here and throughout the article, for $Y \subset \BP^N$, the \emph{$r$-th secant variety} $\s_r(Y)$ is defined
as
\begin{equation*}
\s_r(Y)     = \overline{\bigcup_{y_1\hd y_r\in Y}  \langle y_1\hd y_r\rangle}\subset \BP^N
\end{equation*}
where $\langle y_1\hd y_r\rangle\subset \BP^N$ denotes the linear span of the points $y_1\hd y_r$ and the overline
denotes Zariski closure.
The \emph{$d$-th Veronese embedding} is denoted $v_d \colon \PP V \to \PP(S^d V)$.

Since  $\sigma_r(v_d (X))$ is contained in $\sigma_r(v_d (\PP V))$
  and also in $\langle v_d (X) \rangle$, the linear span of $v_d (X)$, it is contained
in the intersection:
\begin{equation}\label{equ_subset}
   \sigma_r(v_d (X)) \subset \sigma_r(v_d (\PP V)) \cap \langle v_d (X) \rangle
\end{equation}

Our first result  is that for smooth $X$ and $d$ sufficiently large, the above inclusion is (set-theoretically) an equality.
Let $Got(h_X)$ denote the Gotzmann number of the Hilbert polynomial of $X$, see
   Proposition~\ref{prop_Gotzmann_regularity}.

\begin{theorem}\label{item_case_smooth_r}\label{thm_eis_reduce_to_PN}
  Let
  $X\subset \BP^n$ be a smooth subvariety and let $r\in \BN$.
  For all $d \ge r-1 + Got(h_X)$,
    one has the equality  of sets
  \[
    \sigma_r  (v_d(X)) = \reduced{\bigl(\sigma_r(v_d (\PP^n)) \cap \langle v_d (X) \rangle\bigr)},
  \]
  where $\reduced{(\cdot)}$ denotes the reduced subscheme.
\end{theorem}

We expect that the equality of sets in the theorem is really an equality of ideals,
   that is the defining ideal of $\sigma_r (v_d(X))$ is equal to the sum of ideals of
   $\sigma_r  (v_d(X))$ and $\langle v_d (X) \rangle$.
 Evidence for the equality of ideals is given by work in progress by Weronika Buczy\'nska, Mateusz Micha\l{}ek and the first named author.

 Theorem \ref{item_case_smooth_r}    is motivated by a question of Eisenbud, which we review in Section~\ref{intro}, and
questions arising in applications comparing, for a projective variety $Z$,
 the  scheme   $\s_r(Z)\cap L$ with
$\s_r(Z\cap L)$ where $L$ is a linear space, see \S\ref{appsintro} below.
We prove Theorem~\ref{thm_eis_reduce_to_PN} in Section~\ref{sect_reduction_to_PN}.
The proof is based on the Gotzmann regularity property --- see Proposition~\ref{prop_Gotzmann_regularity}.
The main new ingredient in the proof is the following lemma.

\begin{lemma}[Main Lemma]\label{lemma_p_is_in_span_of_intersection}
  Let $X\subset \BP^n$ be a subscheme.
  Suppose $d\ge r-1 + Got(h_X)$ and
$R\subset \BP^n$ is a $0$-dimensional scheme of degree at most $r$.
Then $\langle v_d(R) \rangle \cap \langle  v_d(X) \rangle = \langle v_d (R \cap X) \rangle$.
\end{lemma}

The application of the lemma to the proof of Theorem~\ref{thm_eis_reduce_to_PN} is as follows:
If $p \in \bigl(\sigma_r(v_d (\PP^n)) \cap \langle v_d (X) \rangle\bigr)$
  is on a secant $\PP^{r-1}$ spanned by $r$ distinct points
  in  $v_d (\PP^n)$,
   let    $R$ denote the zero dimensional scheme consisting of the $r$ points,
  considered as points in $\PP^n$,
  so that $p \in \langle v_d(R) \rangle$.
Then by  Lemma \ref{lemma_p_is_in_span_of_intersection},  $p \in \langle v_d (R \cap X) \rangle$, that is $p$ is on a secant $\PP^{t-1}$ to $v_d (X)$,
  where $t = \# (R \cap X) \le r$.
More  work   must be done to deal with the case when $p$ is not on an honest secant $\PP^{r-1}$,
  but the main idea is the same.
The issue of the \emph{smoothability} of zero-dimensional schemes comes into the picture,
  and we show that the smoothness hypothesis in Theorem \ref{item_case_smooth_r} is needed:

\begin{theorem}\label{eisfalsesing}
  For any $q$ and $r \ge 2$, there exist irreducible, singular varieties $X\subset \BP V$ such that
    $\s_r(v_d(X)) \ne \reduced{\bigl(\sigma_r(v_d (\PP^n)) \cap \langle v_d (X) \rangle\bigr)}$ as sets,
    and moreover $\s_r(v_d(X))$ is not cut out set-theoretically by equations of degree at most $q$ for all $d \ge 2r-1$.
  Explicit examples of curves with this property  are given in
    \S\ref{section_counterexamplesc} and \S\ref{section_counterexamples_complete_intersection}.
\end{theorem}

A more precise result is stated in Theorem \ref{eis_does_not_hold_for_singular_case_big_r} below.
There we explain what type of singularities are needed to obtain the inequality
$\s_r(v_d(X)) \ne \sigma_r(v_d (\PP^n)) \cap \langle v_d (X) \rangle$,
and what type of singularities are needed to have $\s_r(v_d(X))$ defined by equations of high degrees.

If   $X$ has at worst hypersurface singularities
and $r\leq 2$, we show that the conclusion of Theorem \ref{item_case_smooth_r} still holds, see
Theorem \ref{item_case_hypersurface_sings_r_le_2}.
We also show (Theorem \ref{item_case_union_of_components}) that  it holds \lq\lq locally\rq\rq\ for arbitrary $X$, in the
sense that $\sigma_r ( v_d(X))$ is an irreducible component of $\sigma_r(v_d (\PP^n)) \cap \langle v_d (X) \rangle$.
These results  generalize  essentially verbatim to reducible $X$.

\subsection{Background and history}\label{intro}
D.~Mumford   \cite[p.~32, Thm~1]{MR0282975}   observed
  that if $X\subset \BP V$ is
 a projective   variety, and one takes a sufficiently large Veronese re-embedding
 of $X$, $v_d(X)\subset \BP S^dV$, then
$v_d(X)$ will be cut out set-theoretically by quadrics (in fact quadrics of rank at most four),
and moreover, that if $X$ is smooth,
the ideal of $v_d(X)$ will be generated in
 degree two. P.~Griffiths  \cite[Thm~p.~271]{MR720290}  remarked further  that with $d$ as above, and $X$ smooth,
the embedding  $v_{2d}(X)$  will be  cut out set-theoretically  by the two by two minors of a matrix of linear forms.
These results were generalized to ideal-theoretic equations of minors for arbitrary  schemes
by J.~Sidman and G.~Smith  \cite[Thm~1.1]{sidman_smith_determinantial}.

More generally, let $L_1,L_2$ be ample line bundles on  an abstract  variety $X$.
The map
$$
 \phi_{L_1^d\ot L_2^e} :X  \ra \BP\bigl(H^0(X,L_1^d\ot L_2^e)^*\bigr)
$$
will be an embedding for $d,e$ sufficiently large.
Write
$V_1=H^0(X,L_1^d)^*$, $V_2=H^0(X,L_2^e)^*$, so
there is  a map
$V_1^*\ot V_2^*\ra H^0(X,L_1^d\ot L_2^e)$ given on decomposable elements by
multiplication of sections.
Let $W^*$ denote the image of the map, so there is  an inclusion
$W\subset V_1\ot V_2$, and $\phi_{L_1^d\ot L_2^e}(X)\subset \BP W$.
Under this inclusion,  the image of $X$ lies in the Segre variety $Seg(\BP V_1\times \BP V_2)$ of
rank one elements  intersected with $\BP W$ (see,  e.g.~\cite[pp~513--514]{MR944326}, \cite[\S1.2]{LOver}).
 The ideal of the Segre is generated in degree two by the two by two minors,
i.e.,  $\La 2 V_1^*\ot \La 2 V_2^*$,
so these minors provide equations for $X\subset \BP W$.

In the above setting,  $\s_r(\phi_{L_1^d\ot L_2^e}(X))\subset \s_r(Seg(\BP V_1\times \BP V_2))$, and thus
equations for the latter give equations for the former. With this in mind,
define
\be\label{catideal}
I(Rank_r(L_1^d,L_2^e))\subset \Sym (W^*)
\ene
to be
the ideal generated by the image of the $r+1$ by $r+1$ minors.
Note that in general $I(Rank_r(L_1^d,L_2^e))$  need  not   be radical,  or even saturated.

The following conjecture is due to D.~Eisenbud, J.~Koh, and M.~Stillman   \cite[p.~518, Equation~(*)]{MR944326}.

\begin{conjecture}[EKS conjecture  (1988)] \label{EKSconj}
Let $C$ be a reduced, irreducible curve, let
$L_1,L_2$ be ample line bundles on $C$. Then there exists a \lq\lq good constant\rq\rq\  $r_0$ depending only
on the genus of $C$ and the $L_i$ such that there is an equality of ideals
$$
I(\s_r(\phi_{L_1^d\ot L_2^e}(C))) = I(Rank_r(L^d_1,L_2^e))
$$
for all $r\leq r_0(d, e)$. Moreover $r_0$ tends to infinity as $d,e \ra\infty$.
\end{conjecture}

Conjecture \ref{EKSconj} was proved set-theoretically  in the case $C$ is a smooth curve
in  \cite{MR1272376},  and scheme-theoretically for smooth curves
  in  \cite{Ginensky}.  Moreover,  sharp bounds on $d,e$ were given in terms of the genus
of the curve.
Theorem~\ref{eisfalsesing}
  provides counter-examples to Conjecture~\ref{EKSconj} for singular varieties.

To relate Conjecture  \ref{EKSconj} to the  first    paragraph of this subsection,
take $C\subset \PP V$, $L_1=L_2=\cO_{C}(1)$ and $r=1$.
More generally, if $r\ge 1$,
then $W \subset S^{d+e}V\subset S^dV\ot S^e V$ and the corresponding equations are the
so  called   {\it symmetric flattenings} or {\it catalecticant minors}
studied first by Sylvester, see  \cite{MR1735271} for a history.

Conjecture \ref{EKSconj} was generalized to higher dimensions by D.~Eisenbud in the form of a question:

\begin{question}[Eisenbud's question (unpublished)] \label{question_of_Eisenbud}
Let $X$ be a projective variety, let $L_1,L_2$ be  ample  line bundles on $X$.
Fix $r$. Do there exist infinitely many sufficiently large $d,e$ such that
    $I(\s_r(\phi_{L_1^d\ot L_2^e}(X)))= Rank_r(L_1^d,L_2^e) $?
\end{question}
In light of Theorem~\ref{eisfalsesing}, one should add the hypothesis that $X$ is smooth to
this question. See also  \cite[Conj.~1.2]{sidman_smith_determinantial}  for a similar conjecture.
The question was discussed by D.~Eisenbud many times in conversation and was communicated to us in an informal e-mail.

A result announced by A.~Iarrobino and V.~Kanev in \cite[Cor.~6.36]{MR1735271} provides a negative
   answer to the question of Eisenbud already in the case $X=\pp n$ and  $L_1=L_2=\cO(1)$ for $n\geq 4$
   and $r$ sufficiently large.
This is a consequence of \cite[Thm~6.34]{MR1735271},
   which is quoted from a paper by Y.~H.~Cho and A.~Iarrobino \cite{iarrobino_cho} that was  posted on the arXiv  in 2011.
Additional results in this direction (both affirmative and negative) appear in \cite[Thms~1.1, 1.4]{BuBu}.

One special case where Eisenbud's question has a positive answer comes from work of Geramita and Raicu.
In the case where $L_1=L_2=\cO(1)$ and $r=2$,
   it was known that whenever $d \ge 2$ and $e \ge 2$, the equations of $Rank_2(L_1^d,L_2^e)$ cut out
$\s_2(v_{d+e}(\pp n))$  scheme-theoretically.
It was further known that these equations, plus the equations of $Rank_2(L_1^{1},L_2^{d+e-1})$,
   generated the entire ideal of this secant variety.
A.~Geramita conjectured in \cite[p. 155]{catalecticant}
   that the second collection of equations were superfluous,
    and this was proven by Raicu in \cite[Thm~5.1]{raicugera}.

In  \cite[Conj.~1.2]{sidman_smith_determinantial},  a slightly different form of the question is stated as a conjecture that
involves only one line bundle which is required to be sufficiently ample.

The Eisenbud-Koh-Stillman conjecture and the question of Eisenbud were stated in the ideal-theoretic setting, i.e.,
  the two  schemes  in question had the same ideals.
One could attempt to address the weaker scheme, or set-theoretic problems.
Another weaker form of the problem would be simply to determine,
   if the ideal of $\s_r(\phi_{L_1^d\ot L_2^e}(X))$ is generated in degree $r+1$, or
   even weaker, that   $\s_r(\phi_{L_1^d\ot L_2^e}(X))$ is cut out set-theoretically by equations of  degree $r+1$.
It is this last statement, in the special case where one begins with $X\subset \BP V$ and
   only considers Veronese re-embeddings, i.e., $L_1=L_2=\cO_{\BP V}(1)|_X$,
   that came to our attention because of its connections
   with conjectures originating in signal processing and theoretical computer science
   that we explain below.
By Theorem \ref{eisfalsesing}, we should restrict attention
   to smooth varieties.
Thus we focus on the following special case:

   Let $X\subset \BP V$ be a smooth
       variety,   and fix $r\in \BN$.
   Do   there exist infinitely many $d$ such that
      $ \s_r(v_d(X)) $ is cut out set-theoretically by equations of  degree $r+1$?

Theorem \ref{item_case_smooth_r} implies
   that the answer to this question is affirmative
   if it is affirmative for $X=\BP V$.

Thus it remains to resolve the following question:

\begin{question}\label{eis_for_PN}
    Let $V=\BC^{n+1}$ and fix a natural number $r$.
    Does there exist an integer $d_0=d_0(n,r)$,
    such that for infinitely many (or even all)  $d \ge d_0$,
    there exists an ideal $I \subset \Sym (S^d V^*)$ generated in degrees at most $r+1$,
    such that the (reduced) subvariety in $\BP (S^d V)$
    consisting of  the zero locus of $I$ is $\sigma_r(v_d (\BP V))$?
\end{question}

The equations of secant varieties  of  Veronese embeddings of $\BP V$ are studied intensively, see  \cite{LOver}  for
the state of the art in spring 2011.
The strongest result related to Question~\ref{eis_for_PN} is in \cite[Thm~1.1]{BuBu}.
There W.~Buczy\'nska and the first named author proved that for $r\leq d,e$ and either $r\leq 10$ or $n\leq 3$,
  the catalecticant minors \eqref{catideal} are enough to define $\s_r(v_{d+e}(\pp n))$ set-theoretically,
      and thus in these cases Question~\ref{eis_for_PN} has an affirmative answer.
In general, $\sigma_r(v_d (\BP V))$ is a component of a Rank locus
   for $r\leq \binom{\lfloor \frac d2\rfloor +n}n$, see  \cite[\S 1.3]{LOver}, \cite[Thms~4.5A, 4.10A]{MR1735271}.

\begin{remark}
There is little  information known about ideals of secant varieties:
 for any nondegenerate variety, the ideal of its $r$-th secant variety
 is empty in degree $r$ \cite[Lemma 2.2]{LM0} and for certain special examples,
e.g.~{\it sub-cominuscule varieties} (see   \cite[\S 5]{LMseries}) which include quadratic
 Veronese varieties and two-factor Segre varieties,
   the ideal is known
   to be generated in degree $r+1$ for all $r$.
\end{remark}

\subsection{How we were led to these questions}\label{appsintro}
For many applications, one needs defining equations for  secant varieties to Segre and Veronese varieties.
A typical problem that arises in applications is  as follows:  one is handed a tensor and needs
to decompose it into  a minimal sum of rank one tensors. It is natural to generalize
this decomposition to arbitrary varieties as follows:

\begin{definition}
  Let $X \subset \BP V$ be a  reduced scheme  and let $p \in \langle X \rangle$.
  \begin{itemize}
    \item Define $R_X(p)$ (the \emph{$X$-rank of $p$})
          to be the minimal number $r$, such that $p \in \langle p_1\hd p_r\rangle$ for some points $p_i \in X$.
    (Note that  $\sigma_r(X) \subset \BP V$
          is the closure of the set of points in $\langle X \rangle$ of $X$-rank at most $r$.)
    \item Define $\ur_X(p)$ (the \emph{$X$-border rank of $p$})
          to be the minimal number $r$, such that $p \in \sigma_r(X)$.
  \end{itemize}
\end{definition}

When $X$ is a Segre or Veronese variety, the $X$-rank is the smallest number $r$ of
rank one terms needed for a decomposition, and the $X$-border rank is the smallest
$r$ needed if one is statisfied with a decomposition accurate within an $\ep$ of one's choosing.

We consider $R_X$  and   $\ur_X$   as   functions  $\langle X \rangle \to \BN$
and if $L\subset \langle X \rangle$, then $R_X|_{L}$ and $\ur_X|_{L}$
denote  the restricted functions.

\begin{definition}\label{brppdef}
 Let $X\subset \BP V$ be a variety
 and $L\subset \BP V$ be a linear subspace.
 Let $Y := \reduced{(X\cap L)}$.
 \begin{itemize}
  \item We say $(X,L)$ is a \emph{rank preserving pair}
        or \rpp{} for short,
        if $\langle Y \rangle = L$
        and $R_{X}|_{L} = R_{Y}$ as functions.
  \item We say $(X,L)$ is a  {\it border rank preserving pair}
        or \brpp{} for short,
        if $\langle Y \rangle = L$
        and $\ur_{X}|_{L} = \ur_{Y}$ as functions,
        i.e., $\sigma_r(X) \cap L=\sigma_r(Y)$ for all $r$.
  \item Similarly we say $(X,L)$ is a \rpp[r]{} (respectively, a \brpp[r]{})
        if $R_{X}(p)=R_{X\cap L}(p)$ for all $p\in L$ with $R_X(p)\leq r$
        (respectively, $\sigma_s(X) \cap L=\sigma_s(Y)$ for all $s \leq r$).
 \end{itemize}
\end{definition}
Note that one always
has $R_{X}(p)\leq R_{X\cap L}(p)$ and $\ur_{X}(p)\leq \ur_{X\cap L}(p)$.

Theorem \ref{thm_eis_reduce_to_PN} may be rephrased in this language:

\begin{theorem}[rephrasing of Theorem \ref{thm_eis_reduce_to_PN}]
  For all smooth subvarieties
  $X\subset \BP V$ and all $r\in \BN$,
  there exist  an integer $d_0$
  such that for all $d \ge d_0$,
  the pair $(v_d ( \BP V),  \langle v_d(X) \rangle)$ is a \brpp[r].
\end{theorem}

\subsection*{Strassen's conjecture}
In complexity theory one is interested in finding upper and lower
bounds for the number of operations required to execute a bilinear
map. One is  especially interested in  the particular
bilinear map matrix multiplication. V. Strassen   \cite[p.~194, \S4, Vermutung~3]{MR0521168}
asked
if there exists an algorithm that simultaneously computes two different
matrix multiplications, that costs less than the sum of
the best algorithms for  the individual matrix
multiplications.  If not, one says that  {\it additivity} holds for matrix
multiplication. Similarly, define additivity for arbitrary bilinear
maps.
 \begin{conjecture}[Strassen]  \cite[p.~194, \S4, Vermutung~3]{MR0521168}  \label{strassenconj} Additivity holds
 for   bilinear maps.
 \end{conjecture}

This may be rephrased as:

\begin{conjecture}[Strassen]\label{conj_strassen_rpp}
Let $A_j$ be vector spaces
  Write $A_j=A_j'\op A_j''$ and let
$  L=\BP ((A_1'\otc A_k')\op (A_1''\otc A_k''))$
Then
$$(X,L)=( Seg(\BP A_1\ctimes \BP A_k), \BP((A_1'\otc A_k')\op (A_1''\otc A_k'')))
$$ is a \rpp.
\end{conjecture}

 \subsection*{Comon's conjecture}  In signal processing one is interested  in
 expressing a given tensor as sum of a minimal number of decomposable tensors.
 Often the tensors that arise have symmetry or at least partial symmetry. Much
 more is known about symmetric tensors than general tensors so it would be convenient
 to be able to reduce questions about tensors to questions about symmetric
 tensors. In particular, if one is handed a symmetric tensor  which has symmetric rank
 $r$, can it have lower rank as a tensor?
 \begin{conjecture}[Comon]\label{comonconj}
  \cite{Como00:ima}
The tensor rank of a symmetric tensor equals its
symmetric tensor rank. That is, for $p\in \BP S^dV$, considering $S^dV\subset V^{\ot d}$,
$R_{v_d(\BP V)}(p)= R_{Seg(\BP V\ctimes \BP V)}(p)$.
\end{conjecture}

This may be rephrased as:

\begin{conjecture}[Comon]\label{conj_comon_rpp}
Let $\tdim A_j=\mathbf{a}$ for each $j$ and identify each $A_j$ with a vector space $A$.
Consider $L=\BP (S^k A)\subset A_1\otc A_k$.
Then
$$(X,L)= (Seg(\BP A \ctimes \BP A ), \BP  (S^k A))
$$ is a \rpp{}.
\end{conjecture}

Our project began with the idea to study these problems simultaneously.
We have included a discussion of these related conjectures in the hope of bringing them
to the attention of the community of algebraic geometers.
A few  general results on rank and border rank preserving pairs
are given in  \S\ref{brppintro}.
  Other results have   appeared elsewhere:
   \cite[Rem.~2.3]{BLtensor}, \cite[Thm~7.1]{BLtensor},  \cite[Cor.~1.10]{BLsecant}, and
   \cite[Rem.~2.3]{CCG_monomials} or \cite[Exercise~3.2.2.2]{Ltensor}.

\subsection*{Border rank versions}

In the cases of the conjectures of Comon (\ref{conj_comon_rpp})
   and Strassen (\ref{conj_strassen_rpp}), it
is natural to ask the corresponding questions for border rank.
For Strassen's conjecture, this has already been answered negatively:

\begin{theorem}[Sch\"onhage] \cite{MR623057}  The pair
$$
(X,L)=(Seg(\BP A\times \BP B\times \BP C), (A'\ot B'\ot C')\op (A''\ot B''\ot C''))
$$
  is not a \brpp{}
starting with the case
$\dim A \geq 5=2+3$, $\dim B\geq 6=3+3$, $\dim C\geq 7=6+1$ , where the splittings
into sums give the dimensions of the  subspaces $A', A''$, etc.
\end{theorem}

See \cite[\S 11.2]{Ltensor} for a discussion of   Sch\"onhage's theorem.

\subsection{A more precise version of Theorem \ref{eisfalsesing}}\label{moreprecisesect}
For a  reduced scheme  $X\subset \BP V$ and a point $x\in X$, the {\it tangent star} of $X$ at $x$, $T^{\star}_xX\subset \BP V$ is defined to be the union of the  points on the $\pp 1$'s
obtained as limits in the Grassmannian $\BG(1, \BP V)$ of $\BP_{x(t),y(t)}^1$'s  spanned by points
$x(t),y(t)$, with $x(t),y(t)\in X$ and $x(0)=y(0)=x$. Alternatively, consider the
incidence correspondence
 \[
  S_X:= \overline{\{ (x,y,z)\in X\times X\times \BP V\mid z\in \langle x,y\rangle\}},
 \]
 let
$\psi: S_X\ra X\times X$, $\mu: S_X\ra \BP V$ denote the projections, then
$T^{\star}_xX=\mu(\psi\inv(x,x))$.
   Note that if $x\in X$ is a smooth point, then
$\langle T^{\star}_xX\rangle$ is the embedded tangent projective space and
$T^{\star}_xX=\langle T^{\star}_xX\rangle$ (but the converse does not hold).

\begin{definition}\label{def_componently_defines}
   Let $I \subset \Sym (V^*)$ be a homogeneous ideal, and suppose $Y \subset \PP V$ is a reduced subscheme.
   Let $Z = Z(I)\subset \PP V$ be the scheme defined by $I$ and let $\reduced{Z}$ be the reduced subscheme.
   We say \emph{$I$ componently defines $Y$}, if $Y$ is a union of some of the irreducible components of $\reduced{Z}$.
   In case $Y$ is irreducible, this just means $Y$ is an irreducible component of $\reduced{Z}$.
\end{definition}

\begin{theorem}\label{eis_does_not_hold_for_singular_case_big_r}
   Let $X \subset \BP V$ be a  subvariety and let  $x \in X$ be a
singular point. Suppose $r \ge 2$ and let $d \ge 2r-1$.
   \begin{enumerate}
    \item  If $T^{\star}_xX \ne \langle T^{\star}_x X\rangle $, then
           $\sigma_r ( v_d(X)) \ne \reduced{\bigl(\sigma_r(v_d (\PP^n)) \cap \langle v_d (X) \rangle\bigr)}$.
    \item  Suppose $I(T^{\star}_xX)$ is the defining ideal of the
tangent star in $\langle T^{\star}_xX\rangle$.
           Suppose for some $q$, the homogeneous parts
             $\bigoplus_{i =0}^{q} I_i(T^{\star}_xX)$ componently define $T^{\star}_xX$
           (this happens for instance if the ideal $I(T^{\star}_xX)$ is
trivial in degrees $\le q$ and $T^{\star}_xX \ne \langle T^{\star}_xX\rangle$).
           Then $\sigma_r (v_d(X))$ is not defined
set-theore\-ti\-cal\-ly
             by equations of degree $\le q$.
   \end{enumerate}
\end{theorem}

We prove Theorem \ref{eis_does_not_hold_for_singular_case_big_r} in \S\ref{section_counterexamplesp}.

\subsection{A uniqueness theorem for symmetric tensor rank}

The following theorem is a consequence of Lemma \ref{lemma_p_is_in_span_of_intersection}.
On one hand, it may be viewed as a generalization of the theorem of
Comas and Seguir   \cite{CS}   that states if the rank of a point
in $\BP (S^d\BC^2)$ is larger than its border rank, and the border
rank is small, the rank must be at least $\lfloor \frac d2\rfloor +2$. (Their theorem gives
more precise information about ranks.)
On the other hand, it also gives a criterion for uniqueness of an expression
of a point as a sum of $d$-th powers that does not rely on a general point
assumption   (e.g.~\cite{MR1859030}, \cite{MR2225496})  or a Kruskal-type test  \cite{MR0444690}.

\begin{theorem}\label{uniquenesscora}
   Let $p\in \BP S^dV$.
   If $R_{v_d(\BP V)}(p)\leq \frac {d+1} 2$,
     i.e., the symmetric tensor rank of $p$ is at most  $\frac {d+1} 2$,
     then $R_{v_d(\BP V)}(p)=\ur_{v_d(\BP V)}(p)$
     and the expression of $p$ as a sum of $R_{v_d(\BP V)}(p)$
     $d$-th powers is unique (up to trivialities).
\end{theorem}

 We prove  Theorem \ref{uniquenesscora}  in \S\ref{sect_reduction_to_PN}.

\begin{remark}
  In contrast to the Veronese case,   such a
  result does not hold  for Segre varieties $Seg(\BP V_1\ctimes \BP V_k)\subset \BP (V_1\otc V_k)$.
  When $k=2$ rank equals border rank and there is no uniqueness, and when $k>2$ elements of border rank two
  can have rank $2,3\hd k$.
\end{remark}

\subsection*{Overview}
  In \S\ref{possect} we first review facts about Hilbert schemes, including Gotzmann's regularity result,
and explain the relationship between the study of the smoothable component of the Hilbert scheme with
the study of secant varieties.
We then give the proofs of results concerning smooth varieties.
In \S\ref{section_singular} we discuss the cases of singular varieties,
  first positive results and then we construct explicit counter examples to the EKS conjecture.
We conclude,  in \S\ref{brppintro}, with a brief collection of examples   of $(b)\rpp$'s
  and examples of pairs which are not $(b)\rpp$.

\subsection*{Acknowledgments}

We thank Daniel Erman, Kyungyong Lee, Zach Teitler and Gavin Brown for their hints
  and patient  help with   scheme theory,
  particularly zero-dimensional schemes and their Hilbert schemes, David Eisenbud for
  discussing  his question with us,
  Giorgio Ottaviani for his suggestion to write down a complete intersection example,
  Frank-Olaf Schreyer and Vivek Shende for useful discussions,
   in particular references regarding Gorenstein schemes,
   and for suggesting the equivalence in Lemma~\ref{lemma_equivalence_of_Gorenstein}.
We also thank Roland Abauf and the referee for many useful suggestions to improve the exposition.

\section{Proofs of positive results}\label{possect}

\subsection{Hilbert schemes and regularity}

Let $X \subset \BP V$ be a subscheme.

\donote{Notation.}

  \begin{itemize}
    \item   $\langle X \rangle\subset  \BP V$ denotes the   scheme-theoretic  linear span of $X$.
    \item   $\reduced{X}$ denotes the reduced subscheme of $X$.
    \item  $I(X) \subset \Sym (V^*) $   denotes the homogeneous, saturated ideal defining of $X$.
          The $d$-th homogeneous piece of $I(X)$ is denoted $I_d(X) \subset S^d V^*$.
          The ideal sheaf of $X$ is denoted by $\ccI_X \subset \ccO_{\BP V}$,
          so that $H^0( \ccI_X(d)) = I_d(X)$.
    \item For a positive integer $d$ the \emph{$d$-th Veronese reembedding} of $X$,
          denoted  $v_d(X) \subset \BP S^d V$,
          is the subscheme defined by the ideal in $\Sym (S^d V^*)$ which is the kernel of the following composition:
          \[
               \Sym(S^d V^*)  = \bigoplus_{k=0}^{\infty} S^{k} (S^d V^*)
               \twoheadrightarrow  \bigoplus_{k=0}^{\infty} S^{kd} V^*
               \hookrightarrow  \bigoplus_{k=0}^{\infty} S^{d} V^*
                   =   \Sym(V^*)
               \twoheadrightarrow   \Sym(V^*) / I(Y).
          \]
  \end{itemize}

Note that for a scheme $X\subset \BP V$ the linear span $\langle X \rangle$ is equal to $\BP ({I_1(X)}^{\perp})$,
 i.e., the projective zero locus of the linear part of $I(X)$.
In particular, $\langle v_d(X) \rangle = \BP ({I_d(X)}^{\perp}) \subset \BP S^d V$.
It is a standard  fact that $v_d (X)$ is isomorphic to $X$ as an abstract scheme,
  see, e.g.,  \cite[Chapter~2, Theorem~2.4.7]{ega2} or \cite[Ex.~II 5.13]{MR0463157}.

For a review on Hilbert schemes, Hilbert polynomials,  Hilbert functions, and regularity see, e.g.,
\cite{nitsure_construction_of_Hilb_and_Quot} and references therein.

Given a projective subscheme $X \subset \BP V$ with ideal sheaf $\ccI_X$,
we say $\ccI_X$ is \emph{$\delta$-regular},
if $H^i(\ccI_X(\delta - i)) =0$  for all $i>0$.
Serre's vanishing theorem implies that every $X \subset \BP V $ has a $\delta$-regular ideal sheaf
for sufficiently large $\delta$.

\begin{proposition}\label{prop_delta_regular}
Suppose $\ccI_X$ is $\delta$-regular with $\delta \ge 0$.
  \begin{enumerate}
     \item \label{item:delta_regular_then_d_regular}
           $\ccI_X$ is also $d$-regular for all $d \ge \delta$.
     \item \label{item:delta_regular_then_HiOX_vanish}
           $H^i(\ccO_X(d)) =0$ for $d \ge \delta-i$  and $i>0$.
     \item \label{item:delta_regular_then_H0_equals_Hilb_poly}
           If $h_{X}$ is the Hilbert polynomial of $X$,
           then $h^0(\ccO_{X}(d)) = h_{X}(d)$ for all $d \ge \delta-1$.
  \end{enumerate}
\end{proposition}
\begin{proof}
   Part \ref{item:delta_regular_then_d_regular} is explained in \cite[Lem.~2.1(b)]{nitsure_construction_of_Hilb_and_Quot}
    with $\ccF = \ccI_X \subset \ccO_{\PP V}$.
   Part~\ref{item:delta_regular_then_HiOX_vanish} follows from the long exact cohomology sequence of
   \begin{equation} \label{equ_ideal_ring_short_exact_sequence}
       0 \to \ccI_X(d) \to \ccO_{\BP V} (d) \to \ccO_X(d) \to 0.
   \end{equation}
   Part~\ref{item:delta_regular_then_H0_equals_Hilb_poly} follows from \ref{item:delta_regular_then_HiOX_vanish}, keeping
in mind that the Hilbert polynomial is also an Euler characteristic.
\end{proof}

Gotzmann's regularity theorem gives a bound on how large $\delta$ must be for $\ccI_X$ to be $\delta$-regular.
This bound depends only on the Hilbert polynomial of $X$, which is
  essential for our purposes.

\begin{proposition}[Gotzmann's regularity, {\cite{MR0480478}}]\label{prop_Gotzmann_regularity}
           Suppose $P$ is the Hilbert polynomial of a subscheme $X \subset \BP V$.  Then
there exists a unique natural number $Got(P)$  such that
           \[
             P(d) = \sum_{i=1}^{Got(P)} \binom{d + a_i - i +1}{a_i}
           \]
           for some $a_1 \ge a_2 \ge .. \ge a_{Got(P)} \ge 0$.
           Moreover $\ccI_X$ is $Got(P)$-regular.
           In particular:
           \begin{itemize}
            \item  if $X' \subset \BP V$ is another scheme with the same Hilbert polynomial $P$,
                   then $\ccI_{X'}$ is also $Got(P)$-regular.
            \item  if $R \subset \BP V$ is a zero-dimensional scheme of degree $r$,
                   then $\ccI_R$ is $r$-regular.
           \end{itemize}
\end{proposition}
The   number $Got(P)$  is called the {\it Gotzmann number} of $P$.
For an exposition of the proof,  see   \cite[Thm~4.3.2]{bruns_herzog_CM_rings} or \cite[Thm 3.11]{MR1648665}.

\begin{lemma}\label{lemma_dimension_of_linear_span}
  Suppose $X \subset \BP V$ is a  subscheme with Hilbert polynomial $h_X$ and $\ccI_X$ is $\delta$-regular.
  Then for $d\ge \delta -1$ and $d >0$:
  \[
     \dim \langle v_d(X) \rangle +1  = h^0(\ccO_X(d)) = h_X(d).
  \]
  In particular, if $R$ is a zero-dimensional scheme of degree $r$, and $d \ge r-1$
  then $\dim \langle v_d(R) \rangle = r-1$.
\end{lemma}
\begin{proof}
  Since all the higher cohomologies vanish,
    the short exact sequence \eqref{equ_ideal_ring_short_exact_sequence}
    gives rise to a short exact sequence of sections.
  The codimension of $\langle v_d(X) \rangle$ in $\PP(S^d V)$
    is equal to $h^0 (I_d(X))$.
  Thus $\dim \langle v_d(X) \rangle+1 = h^0(\ccO_{\PP V}(d)) - h^0 (I_d(X)) = h^0(\ccO_X(d))$
    and the claim follows.
  The $+1$ is just the difference between projective and vector space dimensions.
\end{proof}


The following lemma is elementary, but we include a complete proof for the benefit of readers whose main background is outside of algebraic geometry.

\begin{lemma}[Additivity of Hilbert polynomials]\label{lemma_additivity_of_Hilbert_polynomials}
   Suppose $X, R \subset \BP V$ are two subschemes.
   Suppose $h_{X}$, $h_{R}$, $h_{X\cap R}$ and $h_{X\cup R}$ are respectively the  Hilbert polynomials
   of ${X}$, ${R}$, ${X\cap R}$ and ${X\cup R}$. Then:
   \[
      h_{X \cap R} = h_{X} + h_{R} - h_{X \cup R}.
   \]
   If in addition $R$ is zero-dimensional,
     then $X \cup R$ is $(Got(h_X)+t)$-regular where $t = \deg R - \deg (X \cap R)$.
\end{lemma}
\begin{proof}
   For $d$ sufficiently large, all the higher cohomologies vanish in the following
   short exact sequence:
   \[
     0 \to \ccO_{X \cup R}(d) \to \ccO_{X} (d) \oplus \ccO_{R}(d) \to  \ccO_{X \cap R}(d) \to 0.
   \]
   The additivity claim follows.

   To see the second claim,
   let
   \[
      h_{X}(d) = \sum_{i=1}^{Got(h_X)} \binom{d + a_i - i +1}{a_i}
   \]
   as in Proposition~\ref{prop_Gotzmann_regularity}.
   Set $a_{Got(h_X)+1} = \dotsb =  a_{Got(h_X)+t} = 0$ and note:
   \begin{align*}
      h_{X \cup R}(d)  =  h_X(d) + t
                      & = \sum_{i=1}^{Got(h_X) \phantom{+t}} \binom{d + a_i - i +1}{a_i} + t \cdot \binom{d + 0 - i +1}{0}\\
                      & = \sum_{i=1}^{Got(h_X)+t} \binom{d + a_i - i +1}{a_i}.
   \end{align*}
   Thus by uniqueness of $Got(h_{X \cup R})$ in Proposition~\ref{prop_Gotzmann_regularity},
     we must have $Got(h_{X \cup R}) = Got(h_X)+t$.
\end{proof}

For a projective  reduced scheme  $X$,
  let $\ccH_r(X)$ denote the union of all the irreducible components of
the Hilbert scheme $Hilb_r(X)$ of degree $r$ dimension $0$ subschemes of $X$,
which contain $r$ distinct points.
In case $X$ is  a variety,  $\ccH_r(X)$ is irreducible too.
Also if $Y\subset X$, then $\ccH_r(Y) \subset \ccH_r(X)$.
Schemes that are in $\ccH_r(X)$ are called \emph{smoothable in $X$}
 (because there exists a flat irreducible deformation to a smooth scheme).
It is an interesting and non-trivial problem to determine when $Hilb_r(X) = \ccH_r(X)$,
and to identify the schemes that are in $\ccH_r(X)$ if the equality does not hold---
see, e.g.,  \cite{cartwright_erman_velasco_viray_Hilb8}, \cite{erman_velasco_syzygetic_smoothability}
and references therein.

\begin{proposition}\label{prop_smothable_in_X_iff_smoothable_in_Y}
  Suppose $R$ is a zero-dimensional scheme of finite length $r$  and $X$ and $Y$ are two smooth projective varieties.
  If $R$ can be embedded in $X$ and in $Y$, then $R$ is smoothable in $X$ if and only if $R$ is smoothable in $Y$.
\end{proposition}
See \cite[Prop.~2.1]{BuBu}. Alternatively, analogous statements are
   \cite[Lem.~2.2]{casnati_notari_irreducibility_Gorenstein_degree_9},
   or \cite[Lem.~4.1]{cartwright_erman_velasco_viray_Hilb8},
   or \cite[p.4]{artin_deform_of_sings}
   or \cite[p.2]{erman_velasco_syzygetic_smoothability}.

The smoothable component $\ccH_r(X)$ is relevant to our study because of its relation
to secant varieties:

\begin{lemma}\label{drpoints}
        Let  $X \subset \PP V$ be a  reduced scheme  that is not a set of less than  $  r$ points,  and let $d \ge r-1$.
        Let $p \in \PP S^d V$.
        Then  $p \in \sigma_r(v_d (X))$
        if and only if there exists a scheme $R \in \ccH_r (X)$
        such that $p \in \langle v_d(R) \rangle$.
\end{lemma}

\begin{proof}
    By Lemma~\ref{lemma_dimension_of_linear_span},
     the $d$-th Veronese re-embedding
     of any zero dimensional scheme $R$ of degree at most $r$,
     $v_d(R)$   will span a $(r-1)$-dimensional linear (projective) subspace.
   With this in mind, the  claim becomes \cite[Prop.~11]{BGI}.
     See also \cite[Prop.~2.7]{BuBu}.
\end{proof}

\subsection{Proofs of the  main lemma and the uniqueness theorem} \label{sect_reduction_to_PN}

Fix an integer $r$ and a subscheme $X \subset \BP V$.
Let $d_0 = Got(h_X)+r-1$.
Then,
 if $R \subset \BP V$ is a zero-dimensional scheme of degree at most $r$,
 $\ccI_X$, $\ccI_R$, $\ccI_{X \cap R}$ and $\ccI_{X \cup R}$   are $(d_0+1)$-regular
       by Propositions~\ref{prop_delta_regular}\ref{item:delta_regular_then_d_regular},
      \ref{prop_Gotzmann_regularity}
       and Lemma~\ref{lemma_additivity_of_Hilbert_polynomials}.

Recall that   Lemma~\ref{lemma_p_is_in_span_of_intersection} states  that for $d \ge d_0$
  the following equality of linear spans holds:
\[
  \langle v_d(R) \rangle \cap \langle  v_d(X) \rangle = \langle v_d (R \cap X) \rangle .
\]

\begin{proof}[Proof of Lemma~\ref{lemma_p_is_in_span_of_intersection}]
  Let $d\ge d_0$, let $R \subset \BP V$ be a subscheme of degree at most $r$.
  Since $\langle v_d (X \cap R) \rangle \subseteq   \langle v_d(X) \rangle \cap \langle v_d(R) \rangle$   trivially holds,  to prove equality,  it is enough to prove that the dimension of the left hand side
  equals   the dimension of the right hand side.

  Since  $\ccI_X$, $\ccI_R$, $\ccI_{X \cap R}$ and $\ccI_{X \cup R}$ are $d+1$-regular,
  it follows from Lemmas~\ref{lemma_dimension_of_linear_span} and \ref{lemma_additivity_of_Hilbert_polynomials} that
  \begin{align*}
    \dim \langle v_d (X \cap R) \rangle
      & = h_{X \cap \scheme}(d) - 1
        = h_{X}(d) + h_{R}(d) - h_{X \cup R}(d) -1  \\
      & = (\dim \langle v_d(X) \rangle +1) + (\dim \langle v_d(R) \rangle +1)
           -  (\dim \langle v_d (X \cup R) \rangle  +1) -1 \\
      & =  \dim \langle v_d(X) \rangle  + \dim \langle v_d(R) \rangle
           -  \dim  \langle v_d ( X) \cup v_d  (R) \rangle  \\
      & = \dim (\langle  v_d(X) \rangle \cap \langle v_d(R) \rangle).
  \end{align*}
\end{proof}

The following is a consequence of Lemma~\ref{lemma_p_is_in_span_of_intersection} (with $X=Q$):

\begin{corollary}\label{cor_R_subset_Q}
  Suppose $p \in \BP S^d V$ and that $R,Q  \subset \BP V$ are two zero-dimensional   schemes,
     such that $p \in \langle v_d(R) \rangle$ and $p \in \langle v_d (Q) \rangle$ for some $d \ge \deg (R \cup Q) - 1$.
  Suppose furthermore, that $R$ is minimal in the following sense:
     for any $R' \subsetneqq R$ we have $p \notin \langle v_d(R ') \rangle$.
  Then $R \subset Q$.
\end{corollary}

\begin{proof}
  Apply Lemma~\ref{lemma_p_is_in_span_of_intersection} with $X=Q$, and $d_0 = \deg (R \cup Q) -1$.
  Thus $p \in \langle v_d (R \cap Q)\rangle$, and by the assumption that $R$ is minimal,
    $R \cap Q =R$.
\end{proof}

\begin{proof}[Proof of Theorem~\ref{uniquenesscora}]
Suppose $p \in \PP S^d V$ is such that $r:= R_{v_d(\PP V)}(p) \le \frac{d+1}{2}$.
Thus there exists a zero dimensional smooth scheme $R \subset \PP V$,
   which is a union of $r$ distinct reduced points,
   such that $p \in \langle v_d (R)\rangle$.
Note that $R$ is minimal in the sense  of   Corollary~\ref{cor_R_subset_Q}.
Let $Q \subset \PP V$ be any other zero-dimensional subscheme such that $p \in \langle v_d (R)\rangle$ and
   $\deg Q \le r$.
    Then  $Q=R$ by Corollary~\ref{cor_R_subset_Q}.

The first claim of the theorem is that $\uR_{v_d(\PP V)}(p) = r$.
Were the border rank smaller than $r$, by Lemma~\ref{drpoints}
   there would exist a  smoothable zero dimensional scheme $Q \subset \PP V$ of length less than $r$,
   such that $p \in \langle v_d (Q)\rangle$, a contradiction.
The second claim of the theorem is that $R$ is unique, which  was proved
 above.
\end{proof}

\subsection{Proof of Theorem~\ref{item_case_smooth_r}}

We start by introducing the following notation.

\begin{notation}\label{def_Sigma_rdX}
  Given a  reduced scheme  $X\subset \BP V$, let
  $$
   \srdx:= \reduced{\bigl(\sigma_r(v_d (\BP V)) \cap \langle v_d(X) \rangle\bigr)}
  $$
  where $\reduced{( \cdot)}$ denotes the reduced subscheme.
\end{notation}

Theorem~\ref{item_case_smooth_r}
states that $\srdx = \sigma_r(v_d (X))$ for $d \ge d_0 = Got(h_X) +r-1$.
 First we reduce the  theorem   to Lemma~\ref{lemma_existence_of_Q} below,
   and then we discuss methods necessary to prove the lemma.

\begin{proof}[Proof of Theorem~\ref{item_case_smooth_r}]
  Suppose $X$ is smooth and
  $ p \in \srdx$,
  so that by Lemma~\ref{drpoints} there exists a zero-dimensional smoothable subscheme $R \subset \PP V$ of degree at most $r$,
  such that $p \in  \langle v_d(X) \rangle \cap \langle v_d(R) \rangle$.
  Lemma~\ref{lemma_p_is_in_span_of_intersection} implies $p \in \langle v_d (X \cap R) \rangle$.
  Since smoothability of a zero-dimensional scheme is a local property,
  the components of $R$ which have support away from $X$ are redundant, in the sense that we can replace $R$
  with the union of only those components of $R$  that have support on $X$.
  Thus, without loss of generality, assume $\reduced{R} \subset X$ and also $\deg R =r$ for simplicity of notation.

  Note that if $R$ is not reduced, then this does not   necessarily  imply that $R\subset X$.
  In fact, it is possible to construct smoothable $R$ such that $X \cap R$ is not smoothable.
  We outline how to construct such an example in a separate note \cite{jabu_example_to_BGL},
     as it is not necessary for the content of this paper.
   Instead    we will construct another smoothable scheme $Q$, such that $X \cap R \subset Q \subset X$ and
    $\deg Q = \deg R$.
  In general $Q$ is not necessarily isomorphic to $R$ as an abstract scheme
    (for instance, $Q$ might have smaller embedding dimension than $R$).
 The existence of $Q$  will follow    from Lemma~\ref{lemma_existence_of_Q} below.
  Thus $p\in \langle v_d (Q) \rangle$ and by Lemma~\ref{drpoints},  $p \in \sigma_r (v_d( X))$ as claimed.
  The other inclusion $\sigma_r (v_d (X) )\subset \srdx$ always holds.
\end{proof}

\begin{lemma}\label{lemma_existence_of_Q}
   Suppose $X \subset \PP V$ is a smooth subvariety and $R \subset \PP V$ is a smoothable zero dimensional subscheme of degree $r$, whose support is contained in $X$.
   Then there exists a zero dimensional smoothable subscheme $Q \subset X$ of degree $r$ containing $R \cap X$.
\end{lemma}

To prove the lemma we will use some elementary analytic methods.
It might be possible to avoid the analytic methods by using formal neighbourhoods instead,
   however there is one missing ingredient which we were not able to find references for
   (see Question~\ref{question_smoothable} below).

\begin{notation}
Let $D \subset \CC$ denote be a small open analytic disk.
Also let  $\widehat{D} := \Spec \CC[[t]]$
     and $\widehat{D}^{\bullet} := \Spec \CC[[t]][t^{-1}]$
     (which is the spectrum of the field of fractions of $\CC [[t]]$).
\end{notation}
A reader with more differential-geometric background may wish to think of $\Spec \CC [[t]]$
    as of a sufficiently small (or infinitesimally small) analytic disk in $\CC$ around $0$ (the disk may get smaller during the arguments)
    with coordinate $t$,
    and $\widehat{D}^{\bullet}  \subset \widehat{D}$ should be thought of as  an
      (infinitesimally small) punctured disk $D \setminus \set{0}$.
 Formally,   as set $\widehat{D}$ consists of two points:
     the closed point $0 \in \widehat{D}$ corresponding to the maximal ideal $(t) \subset \CC [[t]]$,
     and the generic point $\widehat{D}^{\bullet}$ corresponding to the ideal $(0)$.

In  the  algebraic category we have the following lemma,
  which gives a simple criterion for flatness in the case we are interested in.
The lemma is a slight rephrase of \cite[Prop.~III.9.8]{MR0463157}.
In our case $C$ in the lemma will be either a smooth quasiprojective curve or $\widehat{D}$.

\begin{lemma}\label{lemma_criterion_for_flatness}
   Let $C$ be a regular integral scheme of dimension $1$,
      and let $c\in C$ be a closed point.
   Denote by $C^{\bullet}:= C \setminus \set{c}$.
   Let $\ccR \subset \PP V \times C$  be a closed subscheme.
   Suppose for each (not necessarily closed) point $t \in C \setminus \set{c}$
       the  fibre $\ccR_t$ over $t$
       is reduced and consists of $r$ distinct $t$-points.
   Let $R \subset \PP V$ be the scheme such that the special fibre $\ccR_c$ over $c \in C$
       is equal to $R \times \set{0}$.
   In addition let $\tilde{\ccR}:=\overline{\ccR \setminus \ccR_c} \subset \PP V \times C$,
       that is $\tilde{\ccR}$ is the smallest reduced closed subscheme of $\PP V \times C$
       containing $\ccR \setminus \ccR_c$.
   Then:
   \begin{enumerate}
     \item \label{item_tilde_ccR_is_flat}
           $\tilde{\ccR} \to $C is flat;
     \item \label{item_ccR_flat_iff_deg_is_r}
           $\ccR \to C$ is flat, if and only if $\dim R =0$ and $\deg R = r$,
             if and only if $\ccR = \tilde{\ccR}$;
   \end{enumerate}
\end{lemma}

\begin{proof}
   By \cite[Prop.~III.9.8]{MR0463157} and its proof the map $\tilde{\ccR} \to C$ is flat,
     thus \ref{item_tilde_ccR_is_flat} is proved.
   By the same proposition $\ccR$ is flat if and only if $\ccR = \tilde{\ccR}$.
   Thus
    \ref{item_ccR_flat_iff_deg_is_r}
     follows from the observation  that $\tilde{\ccR} \subset \ccR$,
       $\dim \tilde{\ccR}_c =0$,  and $\deg \tilde{\ccR}_c =r$.
\end{proof}

 Next   we explain what we mean by an analytic smoothing of a zero dimensional subscheme $R \subset Y$
  (here $Y$ will either be equal to $X$ or to $\PP V$ from Lemma~\ref{lemma_existence_of_Q}).

\begin{definition}\label{def_analytic_smoothing}
  Consider  a  subset $ \ccR \subset Y \times D$,  closed in the Euclidean topology.

  Suppose:
   \begin{itemize}
      \item $\ccR = \ccR^{(1)} \cup \dotsb \cup \ccR^{(r)}$ is a union of $r$ subsets;
      \item each $\ccR^{(i)}$ analytically locally is a zero set of a collection of holomorphic functions,
      \item each $\ccR^{(i)}$ is mapped biholomorphically onto $D$
            via the restriction of projection $Y \times D \to D$.
      \item for each $t \in D \setminus \set{0}$, the preimage $\ccR_t$ is a collection of $r$ distinct points.
   \end{itemize}
   In such  a  situation, it makes sense to study the special fibre $\ccR_0$ of $\ccR \to D$, as a finite subscheme in
      $\PP V \simeq \PP V \times \set{0}$.
   (One can use the theorems of GAGA, for instance  \cite[Thm~3, p. 20]{serre_GAGA},   but this special case is much easier:
      near every point $x$ of support of the fibre we have $\ccR_0$ defined by an ideal generated
      by a  collection   of holomorphic functions,
      in such a way that a power of the maximal ideal of $x$ is contained in the ideal,
      so the holomorphic functions can be chosen to be  polynomials.)
   We say that $\ccR$ is an \emph{analytic smoothing of a subscheme}  $R \subset Y$ if $\ccR_0 = R \times \set{0}$.
\end{definition}

Given $\ccR$ an analytic smoothing of $R \subset Y$,
  we define a \emph{completion} $\widehat{\ccR}$ of $\ccR$,
  to be the subscheme $\ccR$ of $Y \times \widehat{D}$
  locally defined by the same power series as Taylor expansions of holomorphic functions defining
  $\ccR$  in  $Y \times D$.
Note that the completion has the following properties:
\begin{itemize}
 \item the special fibre $\widehat{\ccR}_{0}$ is $R \times \set{0}$;
 \item the generic fibre $\widehat{\ccR}^{\bullet}$ over $\widehat{D}^{\bullet}$
           is reduced and  consists   of  $r= \deg R$  distinct $\widehat{D}^{\bullet}$-points;
\end{itemize}

\begin{lemma}\label{lemma_analytic_smoothability}
   Suppose $R \subset Y$ is a zero dimensional subscheme of $Y$. The following conditions are equivalent:
   \begin{itemize}
     \item $R$ is smoothable in $Y$;
     \item there exists an analytic smoothing $\ccR \subset Y \times D$ of $R$.
     \item there exists a subscheme $\widehat{\ccR} \subset Y \times \widehat{D}$,
           such that the special fibre $\widehat{\ccR}_0$ is $R \times \set{0}$
           and the generic fibre $\widehat{\ccR}^{\bullet}$ over $\widehat{D}^{\bullet}$
           is reduced and   consists    of $r$ distinct $\widehat{D}^{\bullet}$-points.
   \end{itemize}
\end{lemma}

\begin{proof}
  First suppose $R$ is smoothable in $Y$, so that $R$ is in $\ccH_r(Y)$,
    and there exist a quasiprojective curve $C$ and a map  $C \to \ccH_r(Y)$,
    such that general point of $C$ is mapped to the locus of the Hilbert scheme representing $r$ distinct points
    and a special point $c_0 \in C$ is mapped to $R \in \ccH_r(Y)$.
  Precomposing with the normalisation of $C$, we may assume $C$ is smooth.
  Consider the pullback $\ccR_C$ of the universal family to $C$,
    that is a flat finite map $\ccR_C\to C$, such that $\ccR_C \subset Y \times C$,
    with special fibre $\ccR_C|_{c_0} = R \times \set{c_0}$ and a general fibre consisting of $r$ distinct (reduced) points.
  Note that $\ccR_C$ is reduced by Lemma~\ref{lemma_criterion_for_flatness}.
  Restricting $\ccR_C \to C$ to a small disk $D \subset C$ around $c_0$,
    we obtain an analytic smoothing $\ccR \subset Y \times D$.

  Now suppose $\ccR \subset Y \times D$ is an analytic smoothing.
  Then the completion $\widehat{\ccR}$ has the properties as in the final item.

  For the remaining implication we refer to \cite[Lem.~4.1]{cartwright_erman_velasco_viray_Hilb8}.
  Note that $\widehat{\ccR} \to \widehat{D}$ is flat by Lemma~\ref{lemma_criterion_for_flatness}.
\end{proof}

\begin{proof}[Proof of Lemma~\ref{lemma_existence_of_Q}]
  Let $U_1 \subset \PP V$ and $U_2 \subset X$ be two sufficiently small open analytic neighborhoods of the support of $R$
    (which by our assumption is contained in $X$).
  Since $R\subset \PP V$ is smoothable, by Lemma~\ref{lemma_analytic_smoothability}
    there exists an analytic smoothing  $\ccR \subset \PP V \times D$.
  Without loss of generality, we may assume $D$ is small enough so that $\ccR \subset U_1 \times D$.
  Suppose also  $\pi:U_1 \to U_2$ is a holomorphic fibration such that $\pi|_{U_2} = \id_{U_2}$.
  There are many such fibrations, provided $U_1$ and $U_2$ are sufficiently small.
  Locally around $x \in R$,   one way to obtain them is  by composing the following holomorphic maps:
  \[
     U_1 \to T_x \PP V \to T_x X \to U_2, \text{ where:}
  \]
  \begin{itemize}
   \item $U_1 \to T_x \PP V$ is a biholomorphism of $U_1$ with an open neighborhood of $0$ in $T_x \PP V$,
            such that $U_2$ is mapped into $T_x X$ (this exists by the inverse function theorem, because $X$ is smooth);
   \item $T_x \PP V \to T_x X$ is any linear projection such that the restriction to $T_x X$ is the identity;
   \item $T_x X \to U_2$ is defined in a small analytic neighborhood of $0$
            and is the inverse of $U_1 \to T_x \PP V$ restricted to $U_2$.
  \end{itemize}
  It is clear that we have a choice of linear projections $T_x \PP V \to T_x X$,
     and a fibration $\pi$ that arises from a general such projection will have the following property
     (perhaps after replacing $D$ with a smaller disk):
  \begin{itemize}
   \item the $r$ disjoint components of $\ccR \setminus \ccR_0 = \ccR|_{D \setminus \set{0}}$
            under $\pi \times \id_{D \setminus \set{0}}$ are mapped to $r$ disjoint components in
            $U_2 \times (D \setminus \set{0})$
  \end{itemize}
  Let $\ccQ \subset X \times D$ be the image $(\pi \times \id_{D})(\ccR)$.
  It is straightforward to verify that the properties of Definition~\ref{def_analytic_smoothing}
     are satisfied for the family $\ccQ$, so $\ccQ$ is an analytic smoothing of its special fibre
     $Q \subset X$, which is smoothable by Lemma~\ref{lemma_analytic_smoothability}.
  It remains to verify that $R \cap X \subset Q$.

  Informally speaking, $\pi(R\cap X)= R \cap X$, because $\pi|_{U_2} = \id_{U_2}$ and $R\cap X \subset U_2$.
  Thus:
  \begin{equation} \label{equ_R_cap_X_subset_Q}
           (R\cap X) \times \set{0}  = \pi(R\cap X) \times \set{0} \subset (\pi \times \id_{D})(\ccR)  = \ccQ
  \end{equation}
  and thus $R\cap X \subset Q$.

  However formally $\pi(R\cap X)$ makes no sense, because $R \cap X$ and $\pi$ belong to different categories.
  One way to overcome this is to use the category of analytic spaces, another is to use completions of local rings.
  We explain the latter method.

  Let $\widehat{U_1}$ be the disjoint union of $\Spec \hat {\ccO}_{x,\PP V}$ over closed points $x \in R$,
     and analogously let  $\widehat{U_2}$ be the union of $\Spec \hat {\ccO}_{x,X}$.
  We can ``restrict'' $\pi$ to $\widehat{\pi}: \widehat{U_1} \to \widehat{U_2}$
      using Taylor expansions of the holomorphic functions defining $\pi$.
  Since $R \subset \widehat{U_1}$
      and $\widehat{\pi}|_{\widehat{U_2}} =  \id_{\widehat{U_2}}$,
      we have $\widehat{\pi}(R\cap X)= R \cap X$
      and the formally correct rewrite of \eqref{equ_R_cap_X_subset_Q} is:
  \[
           (R\cap X) \times \set{0}  = \widehat{\pi}(R\cap X) \times \set{0}
                \subset (\widehat{\pi} \times \id_{\widehat{D}})(\widehat{\ccR})  = \widehat{\ccQ}.
  \]
\end{proof}

The proof of Theorem~\ref{item_case_smooth_r} is now complete.
   If the answer to the following question is positive, it would be possible
  to avoid  using analytic methods and argue using the spectra of completions of local rings
  instead of analytic neighborhoods in the proof.
\begin{question}\label{question_smoothable}
  Let $X$ be a smooth variety, let $x\in X$ be a point and let $\widehat{X} = \Spec \hat{\ccO}_{x,X}$.
  Suppose $R \subset X$  is   a zero dimensional subscheme of $X$ supported at $x$.
  If $R$ is smoothable in $\widehat{X}$, then is it necessarily smoothable in $X$?
\end{question}

\section{Extensions to singular varieties}\label{section_singular}

Throughout this section we continue to use Notation~\ref{def_Sigma_rdX}.

For singular varieties the inclusion  $\srdx \subseteq  \sigma_r(v_d (X))$
  may fail  to be an equality
  (see \S\ref{section_counterexamplesc}--\ref{section_counterexamples_complete_intersection}).
In this section we study the inclusion in detail.

\subsection{Properties of tangent star}
We commence with a brief overview of elementary properties of the tangent star defined in \S\ref{moreprecisesect}.

For a  scheme  $X\subset \BP V$ and  a closed point  $x\in X$, the \emph{embedded affine Zariski tangent}
 space $\hat T_xX \subset V$ may be defined by recalling that
 the (abstract) Zariski tangent space $T_xX$ is a linear subspace of $T_x\BP V$, and
 $T_x\BP V=\hat x^*\ot V/\hat x$.
Here $\hat x$ is the $1$-dimensional subspace in $V$ representing $x$,
  and $\hat x^*$ is the dual linear space, so that $\hat x^*=\cO(1)_x$.
The affine Zariski tangent space is the inverse image of $T_xX \otimes \hat x$ in $V$.
The \emph{projective Zariski tangent space} $\BP \hat T_x X\subset \BP V$ is its associated projective space.

\begin{proposition}\label{prop_tangent_star}
   Let $X \subset \PP V$ be a  reduced scheme,  let $v\colon \PP V \to \PP W$ be an embedding (for instance $v = v_d$),
      and let $x \in X$. Then:
   \begin{enumerate}
     \item \label{item_tngt_star_subset_tngt_space}
           $T^{\star}_x X \subset \PP \hat T_x X$
     \item \label{item_derivative_of_embedding}
           The derivative of $v$ determines a linear isomorphism of $\PP V = \PP \hat T_x (\PP V)$ with $\PP \hat T_{v(x)} v(\PP V)$.
     \item \label{item_star_is_mapped_to_star}
           The isomorphism above maps $T^{\star}_x X$ onto $T^{\star}_{v(x)} v(X)$.
     \item \label{item_nonreduced_deg_2}
           A non-reduced subscheme $R \subset X$ of degree $2$ supported at $x$
             is uniquely determined by a line $x \in \ell \subset \PP \hat T_x X$.
     \item \label{item_nonreduced_deg_2_smoothable}
           A scheme $R$ as in \ref{item_nonreduced_deg_2} is smoothable in $X$ if and only if $\ell \subset T^{\star}_x X$.
   \end{enumerate}
\end{proposition}

Properties \ref{item_tngt_star_subset_tngt_space}--\ref{item_star_is_mapped_to_star} are clear.
\ref{item_nonreduced_deg_2} follows from \cite[\S{}VI.1.3]{eisenbud_harris},
   see also \cite[Example~II-10]{eisenbud_harris} for an elementary example.
\ref{item_nonreduced_deg_2_smoothable} follows from the definition of tangent star.

\subsection{Positive results for singular varieties}

First  we prove $\sigma_r(v_d (X))$ is an irreducible component of $\srdx$.

\begin{theorem}
\label{item_case_union_of_components}Suppose $X \subset \PP V$ is  a variety.
           If   $d \ge \tmax\{2r-2,r-1 + Got(h_X)\} $, then $\sigma_r ( v_d(X))$ is an irreducible component of $\srdx$.
\end{theorem}

\begin{proof}
  The  set
      $\sigma_r(v_d(X))$ is irreducible because $X$ is.
  Let $\Sigma$ be an irreducible component of $\srdx$ containing $\sigma_r(v_d(X))$.
  For a general point $p \in \sigma_r(v_d(X))$,
    let $p \in \langle v_d (R) \rangle$, where $R \subset X$ consists of $r$ distinct points,
    and $p$ is not in the span of any of $r-1$ of those points.
  We claim $p \notin  \sigma_{r-1}(v_d(\BP V))$.
  Suppose to the  contrary
    that $Q \subset \BP V$ is a zero-dimensional scheme of degree $\le r-1$,
    such that $p \in \langle v_d (Q) \rangle$.
  Then   Corollary~\ref{cor_R_subset_Q} implies
    $R \subset Q$, a contradiction, since $\deg R > \deg Q$.

  The set of points with $v_d(\BP V)$-rank $r$
    is open in $\sigma_{r}(v_d(\BP V)) \setminus \sigma_{r-1}(v_d(\BP V))$.
  Thus since $\Sigma \subset \sigma_{r}(v_d(\BP V))$,  $p \in \Sigma$,
  $p\not\in\s_{r-1}(v_d(\BP V))$, and $p$ has $v_d(\BP V)$-rank $r$,
    a general point $p'$ in $\Sigma$ also has $v_d(\BP V)$-rank $r$.
  Let $R'$ be the union of $r$ distinct points of $\BP V$ such that
    $p' \in \langle v_d (R') \rangle$.
  By Lemma~\ref{lemma_p_is_in_span_of_intersection},   $p' \in \langle v_d (X \cap R') \rangle$,
    and $X \cap R'$ is smooth (hence trivially smoothable).
  Thus $p' \in \sigma_r(v_d(X))$ and $\sigma_r(v_d(X)) = \Sigma$ as claimed.
\end{proof}

Recall that a variety $X\subset \BP V$ is said to have {\it at most hypersurface singularities}, if
the dimension of its Zariski tangent space at any point is at most one greater than the
dimension of $X$.

\begin{theorem}\label{item_case_hypersurface_sings_r_le_2}
           If $X$ has only hypersurface singularities,
           and $r = 2$,
           then $\sigma_r  (v_d(X)) = \srdx$ for all $d \ge d_0 = Got(h_X) + 1$.
\end{theorem}

We postpone the proof of the theorem until later in this subsection.
Proposition~\ref{prop_eis_reduction_to_PV_most_general} below gives further, technical conditions
  that imply $\sigma_r  (v_d(X)) = \srdx$.

\begin{prop}\label{prop_eis_reduction_to_PV_most_general}
  Suppose $X \subset \BP V$ is a  reduced scheme  and
  $r$ is an integer such that:
  \renewcommand{\theenumi}{\Alph{enumi}}
  \renewcommand{\labelenumi}{\theenumi.}
  \begin{enumerate}
     \item \label{item_condition_on_smoothability_of_R}
          All zero-dimensional subschemes $R \subset X$  of degree $r$ which are smoothable in $\BP V$
          are also smoothable in $X$.
     \item \label{item_condition_on_dim_Hilb_Q}  All (locally) Gorenstein
          zero-dimensional subschemes $Q\subset X$ of degree $q$ with $ q < r$ are smoothable in $\BP V$.
  \end{enumerate}
  \renewcommand{\theenumi}{(\roman{enumi})}
  \renewcommand{\labelenumi}{\theenumi}
  Then $\sigma_r(v_d (X)) = \srdx$ for all $d \ge d_0 = Got(h_X) +r-1$.
\end{prop}

Condition~\ref{item_condition_on_smoothability_of_R} is quite strong.
It holds for smooth $X$, see Proposition~\ref{prop_smothable_in_X_iff_smoothable_in_Y}.
We observe in Lemma~\ref{lemma_when_smoothability_holds}
that  the condition is satisfied in the situation of Theorem~\ref{item_case_hypersurface_sings_r_le_2}.
On the contrary, in \S\ref{section_counterexamplesc}--\S\ref{section_counterexamples_complete_intersection}
we use failure of condition~\ref{item_condition_on_smoothability_of_R} to produce counter-examples to the EKS Conjecture.

On the other hand, condition~\ref{item_condition_on_dim_Hilb_Q} is much milder
 --- we list several cases when it is known to hold in Lemma~\ref{lemma_when_condition_on_Hilb_Q_holds}.
In particular, in the situation of Theorem~\ref{item_case_hypersurface_sings_r_le_2} it holds trivially,
  as here $q \le 1$.
We are also unaware of any situation when $X$ is singular,
  condition~\ref{item_condition_on_smoothability_of_R} is satisfied,
but condition~\ref{item_condition_on_dim_Hilb_Q} fails to be satisfied.
 In fact,
condition \ref{item_condition_on_dim_Hilb_Q} might not be needed.
For example, condition \ref{item_condition_on_dim_Hilb_Q} fails for smooth $X$ with $\dim X \ge 6$ and $r \ge 15$
(i.e., there exist non-smoothable zero-dimensional Gorenstein schemes with embedding dimension
$6$ and of degree $14$, see \cite[Cor.~6.21]{MR1735271} or \cite[\S6]{BuBu})
 yet for smooth $X$ the equality holds in \eqref{equ_subset}.

Before proving Proposition~\ref{prop_eis_reduction_to_PV_most_general}
  we explain the relation of condition~\ref{item_condition_on_dim_Hilb_Q}
 with our problem in the following lemma.

\begin{lemma}
  \label{lemma_equivalence_of_Gorenstein}
  Suppose $Q\subset \PP V$ is a zero-dimensional subscheme of degree $q$.
  Let $d\ge q-1$ be an integer.
  Then the following conditions are equivalent:
  \begin{enumerate}
    \item \label{item_equivalence_Gorenstein}
        $Q$ is (locally) Gorenstein;
    \item \label{item_equivalence_dim_Hilb_0}
        $\dim Hilb_{q-1} Q = 0$;
    \item \label{item_equivalence_span}
        $\langle v_d (Q) \rangle  \ne \bigcup_{Q' \subsetneqq Q}{\langle v_d (Q') \rangle}$.
  \end{enumerate}
\end{lemma}

Schemes satisfying \ref{item_equivalence_Gorenstein} are studied intensively,
see for instance \cite[Chap. 21]{eisenbud}, \cite{CasnatiNotari}, \cite{Casnati_Notari_On_Gorenstein_locus},
 \cite{MR1735271}.
Condition~\ref{item_equivalence_span} says that $Q$ is minimal in the following sense:
for a general $p \in \langle v_d (Q )\rangle $ there exists no smaller $Q' \subset Q$ such that
$p \in \langle v_d (Q' ) \rangle$.
We thank Frank-Olaf Schreyer and Vivek Shende for (independently) pointing out to us the equivalence of
\ref{item_equivalence_Gorenstein} and \ref{item_equivalence_dim_Hilb_0}.

\begin{proof}
  Conditions \ref{item_equivalence_Gorenstein} and \ref{item_equivalence_dim_Hilb_0}
  are local, i.e., they hold for $Q$ if and only if they hold for all connected components of $Q$.
  Thus to prove the equivalence of \ref{item_equivalence_Gorenstein} and \ref{item_equivalence_dim_Hilb_0}
   we may assume $Q$ is supported at one point and hence the structure ring $\ccO_Q$
  is a local algebra of finite dimension over $\CC$.

  Let $\mathfrak{m}$ be the maximal ideal in $\ccO_Q$
   and  $\mathfrak{s}$ be the \emph{socle} of $\ccO_Q$, that is
   the annihilator of $\mathfrak{m}$ in $\ccO_Q$ (see \cite[p.~522]{eisenbud}).
  Now a subscheme of length $n$ is defined by an ideal of dimension
  $q - n$; consequently,
  $Hilb_{q-1} Q = \PP \bigl(\Hom_{\ccO_Q}( \BC,\ccO_Q) \bigr)$.
  Here $\BC = \ccO_Q/\mathfrak{m}$,
   thus the image of any homomorphism $\BC \to \ccO_Q$ is contained in the socle $\mathfrak{s}$.
  On the other hand, any $f \in \mathfrak{s}$ determines a homomorphism $\BC \to \ccO_Q$,
   by sending $1 \mapsto f$.
  Thus $\dim Hilb_{q-1} Q =0$, if and only if
   $\dim \Hom_{\ccO_Q}( \BC,\ccO_Q)=1$ if and only if
   the socle of $\ccO_Q$ is one-dimensional, if and only if $Q$ is Gorenstein
   (see \cite[Prop.~21.5a\&c]{eisenbud}).

  To prove the equivalence of \ref{item_equivalence_dim_Hilb_0} and \ref{item_equivalence_span}
   note that:
  \renewcommand{\theenumi}{(\alph{enumi})}
  \renewcommand{\labelenumi}{\theenumi}
   \begin{enumerate}
    \item $Hilb_{q-1} Q$ is a projective scheme;
    \item if $Q'' \subsetneqq Q$ is a non-trivial subscheme,
           then there exists a subscheme $Q'$ of degree $q-1$ such that
			    $Q'' \subset Q' \subset Q$
          and thus $\bigcup_{Q' \subsetneqq Q}{\langle v_d (Q') \rangle}$
          is the same if we restrict the union to only $Q'$ of degree $q-1$;
    \item if $Q', Q'' \subset Q$ are two subschemes and $Q' \ne Q''$,
          then $\langle v_d (Q') \rangle \ne  \langle v_d (Q'') \rangle$,
          because $d \ge q-1$, see Lemma~\ref{lemma_dimension_of_linear_span};
    \item Since $d\ge q-1$, for all $Q' \subset Q$ (including $Q' = Q$),
           we have $\dim \langle v_d (Q') \rangle = \deg Q' -1$.
  \end{enumerate}
  \renewcommand{\theenumi}{(\roman{enumi})}
  \renewcommand{\labelenumi}{\theenumi}
  Thus, if $\dim Hilb_{q-1} Q =0$,
    then $\dim \bigcup_{Q' \subsetneqq Q}{\langle v_d (Q') \rangle}$ is $q-2$,
    so this union cannot be equal to $\langle v_d (Q') \rangle$.

  On the contrary, if  $\dim Hilb_{q-1} Q > 0$,
    then $\bigcup_{Q' \subsetneqq Q}{\langle v_d (Q') \rangle}$ is swept out by
  a projective, positive dimensional family of distinct linear subspaces of dimension $q-2$,
  thus it is closed and of dimension at least $q-1$.
  Since it is always contained in $\langle v_d (Q) \rangle$ (which is irreducible and of dimension $q-1$),
   it follows that
  \[
    \bigcup_{Q' \subsetneqq Q}{\langle v_d (Q') \rangle} = \langle v_d (Q) \rangle.
  \]
\end{proof}

\begin{proof}[Proof of Proposition~\ref{prop_eis_reduction_to_PV_most_general}]
  Suppose
  \[
     p \in \srdx:= \reduced{\bigl(\langle v_d(X)\rangle \cap \s_r(v_d(\BP V))\bigr)},
  \]
  so that by Lemma~\ref{drpoints} there exists a zero-dimensional smoothable subscheme $R \subset \BP V$
   of degree at most $r$,
  such that $p \in  \langle v_d(X) \rangle \cap \langle v_d(R) \rangle$.
  By Lemma~\ref{lemma_p_is_in_span_of_intersection},
    also $p \in \langle v_d (X \cap R) \rangle$.
  If $X \cap R$ is smoothable in $X$,
    then $p \in \sigma_r(X)$ by Lemma~\ref{drpoints}.
  So suppose $Q$ is the minimal subscheme of $X \cap R$, such that $p \in \langle v_d (Q) \rangle$
    and set $q:=\deg Q$.

  By the minimality of $Q$,
    the hypotheses of Lemma~\ref{lemma_equivalence_of_Gorenstein}\ref{item_equivalence_span} hold for $Q$,
    thus $Q$ is Gorenstein by Lemma~\ref{lemma_equivalence_of_Gorenstein}\ref{item_equivalence_Gorenstein}.
  Now either $Q = R$, and then it is smoothable in $X$ by \ref{item_condition_on_smoothability_of_R},
    or $q < r$, and then $Q$ is smoothable in $\PP V$ by \ref{item_condition_on_dim_Hilb_Q}.
  Note that condition \ref{item_condition_on_smoothability_of_R} also holds  for $R$ replaced with $Q \cup \setfromto{x_1}{x_{r-q}}$,
    where the  $x_i$ are distinct points, disjoint from the support of $Q$.
  Thus $Q$ is smoothable in $X$.
\end{proof}

The   following Lemmas determine  situations when conditions~\ref{item_condition_on_smoothability_of_R}
and \ref{item_condition_on_dim_Hilb_Q} are satisfied.

\begin{lemma}\label{lemma_when_smoothability_holds}
  Suppose $X\subset \BP V$ is  a  reduced scheme  and $r=2$.
         If   $T^{\star}_x X = \BP \hat {T}_x X$ for all $x\in X$
           (for instance, $X$ has at worst hypersurface singularities),
         then condition~\ref{item_condition_on_smoothability_of_R} is satisfied.
\end{lemma}

\begin{proof}
  A scheme $R$ of degree $2$ is either
    a disjoint union of $2$ points (which is trivially smoothable)
  or $R$ is a double point supported at $x \in X$.
  In the second case the result follows from
    Proposition~\ref{prop_tangent_star}\ref{item_nonreduced_deg_2} and \ref{item_nonreduced_deg_2_smoothable}.
\end{proof}

\begin{lemma}\label{lemma_when_condition_on_Hilb_Q_holds}
  Suppose $X\subset \BP V$ is a subscheme and $r \le 11$
    or $X$ can be locally embedded into a smooth $3$-fold.
  Then condition~\ref{item_condition_on_dim_Hilb_Q} is satisfied.
\end{lemma}

\begin{proof}
  If $r \le 11$, then $q\leq 10$, and zero-dimensional Gorenstein schemes of degree at most $10$ are smoothable,
    see the main theorem in \cite{CasnatiNotari}.
  If $X$ can be locally embedded into a smooth $3$-fold,
   then also the embedding dimension of any $Q\subset X$ is at most $3$ at each point,
   and a zero-dimensional Gorenstein scheme that can be embedded in $\PP^3$ is smoothable
   by \cite[Cor.~2.4]{Casnati_Notari_On_Gorenstein_locus}
   and thus $Q$ is also smoothable in $\PP V$ by Proposition~\ref{prop_smothable_in_X_iff_smoothable_in_Y}.
\end{proof}

Theorem  \ref{item_case_hypersurface_sings_r_le_2}
      follows  from Proposition~\ref{prop_eis_reduction_to_PV_most_general} as
  conditions \ref{item_condition_on_smoothability_of_R} and \ref{item_condition_on_dim_Hilb_Q} hold
    by Lemmas~\ref{lemma_when_smoothability_holds} and \ref{lemma_when_condition_on_Hilb_Q_holds}.
We also remark, that both conditions~\ref{item_condition_on_smoothability_of_R} and \ref{item_condition_on_dim_Hilb_Q} hold
  for any $r$ when $X$ is an irreducible curve with at most planar singularities,
  that is $X$ can be locally embedded into a smooth surface
  --- see \cite[Thm~5 and Cor.~7]{altman_iarrobino_kleiman_irreducibility_of_Jacobian}.

Here is another consequence of Corollary~\ref{cor_R_subset_Q}.

\begin{corollary}\label{uniquenesscorb}
 Suppose Condition~\ref{item_condition_on_dim_Hilb_Q} of Proposition~\ref{prop_eis_reduction_to_PV_most_general}
  holds for some $r$ and $X = \BP V$.
  Then for all $p\in \s_r(v_d(\BP V))$, $p\notin \s_{r-1}(v_d(\BP V))$ and $d \ge 2r-1$,
  the scheme $R$ of degree $r$ such that
  $p\in \langle v_d(R)\rangle$ is unique.
\end{corollary}
The hypotheses of Corollary \ref{uniquenesscorb}   hold
 in all dimensions when $r\leq 11$ and for all $r$ when $\dim \PP V \le 3$
 by Lemma~\ref{lemma_when_condition_on_Hilb_Q_holds}.
It fails to hold for when both $r$ and $\tdim(\PP V)$ are large,
  see comments and references after Proposition~\ref{prop_eis_reduction_to_PV_most_general}
\begin{proof}
  Let $R \subset \BP V$ be a smoothable zero-dimensional scheme of degree $r$ such that $p \in \langle v_d(R) \rangle$.
  Suppose $R' \subset R$ is a subscheme such that $p \in \langle v_d(R') \rangle$
    and suppose $R'$ is minimal with this property.
  Condition~\ref{item_condition_on_dim_Hilb_Q} implies that $R'$ is smoothable.
  Since $p\notin \s_{r-1}(v_d(\BP V))$, we have $R' = R$ and $R$ is minimal.
  Thus by Corollary~\ref{cor_R_subset_Q}, the scheme $R$ is unique as claimed.
\end{proof}

\subsection{Explicit examples of curves}\label{section_counterexamplesc}

Let $X\subset \BP V$ be a  reduced scheme.
Recall the incidence correspondence $S_X$ from \S\ref{moreprecisesect} and note that $\s_2(X)=\mu (S_X)$, and
$\tdim T^{\star}_xX\leq 2\tdim X$.

Consider the case $X$ is the union of two lines that intersect in a point $y$. Then
$T^{\star}_yX=\BP \hat T_yX$ is a $\pp 2$.

Now let $X$ be the union of three lines in $\pp 3=\BP V$ that intersect in a point $y$ and
are otherwise in general linear position, e.g., the lines corresponding
to coordinate axes in affine space. (That is, give $\pp 3$ coordinates
$[x_0,x_1, x_2,x_3]$ and take the union of lines
which is
given by $x_ix_j=0$ for all
$1\leq i<j\leq 3$,
  i.e., each line is $x_i=x_j=0$ for some $1\leq i<j\leq 3$.)
Then $T^{\star}_yX$ is the union of three $\pp 2$'s, but $ \langle T^{\star}_yX\rangle =\pp 3$.
Consider $v_d(X)$, for $d \ge 3$.
  If we label the coordinates in $S^dV$ in order
$x_0^d,x_0^{d-1}x_1,x_0^{d-1}x_2,x_0^{d-1}x_3,x_0^{d-2}x_1^2,\hdots$,   then
$\langle T^{\star}_{v_d(y)}v_d(X)\rangle =\langle T^{\star}_{v_d(y)}v_d (\PP V)\rangle $ is the span of the
first four coordinate points and
$T^{\star}_{v_d(y)}v_d(X)$ is the union of the
$\pp 2$'s spanned by the duals of
$x_0^d,x_0^{d-1}x_i,x_0^{d-1}x_j$, $1\leq i<j\leq 3$.
Consider the point $z=[1,1,1,1,0\hd 0]\in
\langle T^{\star}_{v_d(y)}v_d(X)\rangle $.
It  lies in $\s_2(v_d(\BP V))$,
but it is not in $\s_2(v_d(X))$. To prove this, note that the scheme $R$
of degree two defining
  $z$ as an element of $\s_2(v_d(\BP V))$ is unique by
  Corollary \ref{uniquenesscorb}, but $R$ is obtained
by $x_0^d$ and the tangent vector in the direction of
$x_0^{d-1}(x_1+x_2+x_3)$ and the latter is not in
$T^{\star}_{v_d(y)}v_d(X)$.
 So $R$ is not smoothable in $X$ by Proposition~\ref{prop_tangent_star}\ref{item_nonreduced_deg_2_smoothable}.

 Thus $\s_2(v_d(X))$ is not defined by the equations inherited
 from $\s_2(v_d(\pp 3))$. However, it is defined by cubics, namely
 the cubics inherited from  $\s_2(v_d(\pp 3))$ and those defining the
 union of the three $\pp 2$'s in $\BP S^dV$.

Finally let $X_k$ be the union of $k\geq 4$ lines in $\pp 3=\BP V$ that intersect
 at a point $y$ but are otherwise in general linear position.
 Then $T^{\star}_yX_k$ is a union of $\binom k2$ $\pp 2$'s,
and thus is a hypersurface of degree $\binom k2$ in $\langle T^{\star}_yX_k\rangle = \BP V$.
 As above,  $T^{\star}_{v_d(y)}v_d(X_k)$ is also   a union
 of $\binom k2$ $\pp 2$'s, namely the linear spaces
 whose tangent spaces are the images of the tangent
 spaces to the $\pp 2$'s in $T^{\star}_yX_k$ under the
 differential of the Veronese.
 And as above,
 $\langle T^{\star}_{v_d(y)}v_d(X_k)\rangle = \langle T^{\star}_{v_d(y)}v_d(\pp 3)\rangle
   = \PP \hat T_{v_d(y)} v_d (\PP^3)$ will
 be the $\pp 3\subset \BP S^dV$ that they span
 (see Proposition~\ref{prop_tangent_star}\ref{item_star_is_mapped_to_star}),
and a general point of $\langle T^{\star}_{v_d(y)}v_d(X_k)\rangle $ will not be in $\s_2(v_d(X_k))$.
 This provides an explicit construction of a sequence of
  reduced schemes  such that the ideal
 of $\s_2(v_d(X_k))$ has generators in  degree  $\binom k 2$
 for all $d\geq 3$.

 To obtain irreducible varieties, it is sufficient that they locally
 look like the above example near a point $y$. To be explicit,
 take for example,  a rational normal curve $C\subset \pp n$
 (with $n = k+2$)
 and a linear subspace $W \simeq \BP^{k-1} \subset \BP^n$, spanned by $k$ general points on $C$.
 Then choose a general hyperplane $H \simeq \BP^{k-2} \subset W$
 and let $\pi \colon \BP^n \setminus H \to \BP ^{3}$ be the projection away from $H$.
 Then $W$ is mapped to a single point and $X := \pi(C)$
 is a degree $n$ irreducible curve with singularity isomorphic to $k$ general intersecting lines
 and for any $d \ge 3$, one needs equations of degree at least $\binom k 2$
 to define $\sigma_2 (v_d (X))$, even set-theoretically.

In the next section we show  that  for $d$ sufficiently large, the same examples work for all $r$.

In \S\ref{section_counterexamples_complete_intersection} we present further counter-examples,
 which are complete intersections.

\subsection{Proof of Theorem~\ref{eis_does_not_hold_for_singular_case_big_r} }\label{section_counterexamplesp}


Recall that in Theorem~\ref{eis_does_not_hold_for_singular_case_big_r}
 we give conditions on singularities of $X$ that force $\sigma_r(v_d (X)) \ne  \srdx$
 and conditions that force some of the defining equations of
 $\sigma_r(v_d (X))$ to be of high degree.
In the proof we intersect both $\sigma_r(v_d (X))$ and  $\srdx$ with a linear space $W$,
which for $r=2$ is just the projective tangent space at a sufficiently singular point $\BP \hat T_x v_d(X)$.
We show  there is enough of difference between $\sigma_r(v_d (X))\cap W$ and  $ \srdx \cap W$ to prove the theorem.

In the first lemma below we describe $\srdx \cap W $, while in the next lemma we describe $\sigma_r(v_d (X))\cap W$.

\begin{lemma}\label{lemma_T_is_contained_in_varsigma_case_big_r}
  Let  $X \subset \BP V$ be a variety,
    let  $r \ge 2$,  let $d \ge 1$, and let
    $x, \fromto{y_1}{y_{r-2}} \in v_d (X)$ be $r-1$ disjoint points.
  Then
  \[
     W:=\langle \BP \hat T_x v_d(X) \cup \setfromto{y_1}{y_{r-2}}
\rangle \subset \srdx.
  \]
\end{lemma}

\begin{proof}
  By definition $\srdx  = \reduced{(\sigma_r(v_d(\BP V)) \cap \langle
v_d(X) \rangle)}$.
  Note that
  \[
    \BP \hat T_{x}v_d(X) \subset \BP \hat T_{x}  (v_d (\BP V)) \subset
\sigma_2(v_d(\BP V))
  \]
  and thus $W \subset \sigma_r(v_d(\BP V))$.
  Also $W \subset \langle v_d(X) \rangle$.
  Since $W$ is reduced,  the claim follows.
\end{proof}

\begin{lemma}\label{lemma_sigma_r_vdX_has_only_tangent_star_case_big_r}
  In the setup of Lemma~\ref{lemma_T_is_contained_in_varsigma_case_big_r},
    suppose  $d \ge 2r-1$.
  Let $p\in W$ be a point
    which is not contained  in $\langle \BP \hat T_x v_d(X)  \cup
(\setfromto{y_1}{y_{r-2}} \setminus \set{y_i})\rangle $
    for any $i$.
  Then $p\in \sigma_r (v_d(X))$ if and only if
    $p\in \langle x, z,\fromto{y_1}{y_{r-2}}\rangle$
    for some $z \in T^{\star}_{x}v_d(X)$.

   In other words $\reduced{(\sigma_r (v_d(X)) \cap W)}$
    consists of the cone over $T^{\star}_{v_d(x)}v_d(X)$ with vertex
$\langle \fromto{y_1}{y_{r-2}} \rangle$
    and possibly other components contained in
    $\langle \BP \hat T_x v_d(X)  \cup (\setfromto{y_1}{y_{r-2}}
\setminus \set{y_i})\rangle $
    for some $i$.
\end{lemma}
\begin{proof}
  The tangent star is always contained in $\sigma_2 (v_d(X))$, thus one
implication is easy as $\s_r(v_d(X))$ is the join of $\s_2(v_d(X))$ and $\s_{r-2}(v_d(X))$.

  To prove the other implication, suppose $p \in \sigma_r (v_d(X))$.
  If $p \in \langle x, \fromto{y_1}{y_{r-2}} \rangle$, then $z$ can be
taken to be $x$.
  Otherwise, let $R \subset \BP V$ be a smoothable scheme of degree at
most $r$,
  such that $p \in \langle v_d(R) \rangle$ and $R$ is minimal with this
property.
  By the uniqueness provided by Corollary~\ref{cor_R_subset_Q},
    $R = R_z \cup  \setfromto{y_1}{y_{r-2}}$,
    where $R_z$ is the degree $2$ scheme supported at $x$,
    and contained in the line $\langle x, z \rangle$ for some $z \in \BP
\hat T_{x}  (v_d (X))$.
  Since $p \in \sigma_r (v_d(X))$,   $R$ is smoothable in $X$,
    and also $R_z$ is smoothable in $X$.
  So $z$ is in the tangent star of $v_d(X)$ at $x$.
\end{proof}

\begin{proof}[Proof of Theorem \ref{eis_does_not_hold_for_singular_case_big_r}]
  If $T^{\star}_xX \ne \BP \hat T_x X$,
    then Lemmas~\ref{lemma_T_is_contained_in_varsigma_case_big_r} and
\ref{lemma_sigma_r_vdX_has_only_tangent_star_case_big_r}
    imply (i).

 Suppose there do not exist equations of degree at most $q$
  that define componently $T^{\star}_xX$, as in Definition~\ref{def_componently_defines}.
  Thus the same holds for $T^{\star}_{v_d (x)} v_d (X) \subset \BP \hat T_{v_d(x)} v_d(X)$
   by Proposition~\ref{prop_tangent_star}\ref{item_star_is_mapped_to_star}.
  And equations of degree at most $q$ are not enough
  to componently define the join of  $T^{\star}_{v_d (x)} v_d (X)$ with a linear space.

  Fix $r-2$ distinct points $\fromto{y_1}{y_{r-2}} \in X \setminus \set{x}$.
  Let $W$ be as in Lemma~\ref{lemma_T_is_contained_in_varsigma_case_big_r}.

  An ideal $I$ defining $\sigma_r (v_d(X))$
  must contain an ideal $I'$ defining $\srdx$.
  Let $J$ be the ideal generated by linear equations of $W$.
  By Lemmas~\ref{lemma_T_is_contained_in_varsigma_case_big_r} and
\ref{lemma_sigma_r_vdX_has_only_tangent_star_case_big_r},
    $I'+J = J$,
  but $I+J$ componently defines a cone over $T^{\star}_{v_d (x)} v_d (X)$
    in $W$ with vertex $\langle \fromto{y_1}{y_{r-2}} \rangle$.
  Thus $I$ needs more equations than there are in $I'$, and our
   assumptions imply
   that equations of degrees at most $q$ are not enough.
\end{proof}

We conclude that the counter-examples illustrated in
\S\ref{section_counterexamplesc}  also work  for $r \ge 3$.

\subsection{Singular complete intersection counter-examples}\label{section_counterexamples_complete_intersection}

The examples in \S\ref{section_counterexamplesc}  are not   local  complete intersections, and one could try to restrict
the EKS conjecture only to such curves.
Yet, even singular complete intersections fail to satisfy  the conjecture.

Suppose $h$ and $h'$ are two general cubics in three variables $x_1, x_2, x_3$ and let
\[
  f:= x_0 (x_1 x_2 - x_2 x_3) + h \text{ and } f':= x_0 (x_1 x_3 - x_2 x_3) + h'.
\]
In $\PP^3$ consider the scheme $X$ defined by $f=f'=0$.
It is a reduced complete intersection of two cubics, as can be easily verified, for instance, by intersecting with
the hyperplane $x_0=0$.
The curve  $X$ is singular at $x:=[1,0,0,0]$.
The tangent cone at $x$ is given by
\[
 x_1 x_2 - x_2 x_3 =  x_1 x_3 - x_2 x_3=0,
\]
so it is the union of four lines through $[1,0,0,0]$ and one of the four points
$[0,1,0,0]$, $[0,0,1,0]$, $[0,0,0,1]$, $[0,1,1,1]$.
The tangent star in this case is the secant variety of the tangent cone, and thus it is a union of $6$ planes.
Hence it is defined by a single equation of degree $6$ and thus by Theorem~\ref{eis_does_not_hold_for_singular_case_big_r}
the secant varieties $\sigma_r(v_d(X))$ cannot be defined by equations of degree $\le 5$ when $d$ is sufficiently large.

Similarly,   consider $X \subset \PP^3$ to be a complete intersection of:
\[
  f:= {x_0}^s g  + h \text{ and } f':= {x_0}^{s'} g' + h',
\]
where $g,g', h,h'$ are general homogeneous polynomials in $x_1, x_2, x_3$ of degrees, respectively, $t,t', (s+t), (s'+t')$,
with $t, t' \ge 2$.
If $t, t'$ grow, then the degree of the defining equation of the tangent star will
  grow  too.
Thus   one has   complete intersection counter-examples  to the EKS conjecture for arbitrary $r \ge 2$.

\section{rpp and brpp}
\label{brppintro}

\subsection{General facts about rpp and brpp}

Let $\BG(k,\BP V)$ denote the Grassmannian of $\pp k$'s in $\BP V$.

\begin{proposition}
  Suppose $X \subset \PP V$ is  a  non-degenerate  subvariety, $L\in \BG(k,\BP V)$ is general,
  and $\tdim (L)\geq \tcodim (X)$.
  If $(X,L)$ is a $rpp$ then it is a $brpp$.
\end{proposition}

\begin{proof}
   The set of points of $X$-rank at most $r$,
     contains an open subset $U_r$ of $\sigma_r(X)$.
   By our assumptions $U_r \cap L$ is not empty.
   Since $(X,L)$ is     a   \rpp, $U_r$ consists of
     points of $(X\cap L)$-rank at most $r$.
   Moreover, $\sigma_r(X) \cap L$ is the closure of $U_r \cap L$,
      because $\sigma_r(X) \cap L$ is irreducible.
   Therefore  $\sigma_r(X) \cap L \subset \sigma_r(X \cap L)$ and since
      the other inclusion always holds,   it follows $(X,L)$ is a $brpp$.
\end{proof}

Recall that  for a variety $X\subset \BP V$,  $\tdim \s_r(X)\leq r(\tdim X+1)-1$ and
   in a typical situation either $\s_r(X) = \BP V$ or the equality $\tdim \s_r(X) = r(\tdim X+1)-1$ holds.
When neither of these happens, we say  $\s_r(X)$ is \emph{defective},
and we w%
rite $\d_r(X)= r(\tdim X+1)-1-\tdim \s_r(X)$, for the $r$-th \emph{secant defect} of $X$.

\begin{proposition} Let   $X\subset \BP V$ with $\tdim X\geq k$
and assume $\d_r(X)\leq k(r-1)-1$.  Let $L\in  \BG(\tdim V-k-1, \BP V)$ be general.
Then $(X,L)$ is neither   a $rpp$ nor a $brpp$.  \end{proposition}
\begin{proof}
The dimensions have been arranged such that
$\tdim \s_r(X\cap L)<\tdim[\s_r(X)\cap L]$, so there is
$p\in \s_r(X)\cap L$ such that $p\not\in  \s_r(X\cap L)$,
   showing   $(X,L)$ is not a $brpp$.
 Moreover since $L$ is general, it will
have a non-empty intersection with the set of points in $\s_r(X)$ of $X$-rank equal to $r$,
showing   $(X,L)$ is not a $rpp$ either. \end{proof}

\subsection{Examples}
The reader can easily verify the following:

\begin{example}
     Let $X = v_3(\PP^1) \subset \PP (S^3 \CC^2) \simeq \PP^3$,
      and let $L = \PP^2 \subset \PP^3$ be a general plane.
    Then $(X,L)$ is neither a $rpp$ nor a $brpp$.  \end{example}

\begin{example} Let $X\subset \BP V$ be a   (reduced, irreducible)   hypersurface, and let
$L$ be  a linear subspace   such that $\langle \reduced{(X\cap L)} \rangle =L$   (for example $L$ is a general linear subspace of   a  given dimension).  Then $(X,L)$
is both   a $rpp$ and a $brpp$.
\end{example}

\begin{example}
  Let $X = v_2(\pp 2) \subset \pp 5$ and let $L \subset \pp 5$ be a  general hyperplane.
  Then $(X,L)$ is  a $brpp$ but not a $rpp$.

  To see this, note that a general hyperplane section of $v_2 (\pp 2)$ is a
    $v_2(v_2(\pp 1))=v_4(\pp 1)$.  In coordinates, $v_4 (\pp 1)$
  may be described as set of symmetric $3 \times 3$ matrices
  $(x^i_j)$  of rank 1 with $x^1_3=x^2_2$.  The hypersurfaces $\s_2(X)\subset\pp 5$ and $\s_2(X \cap L)\subset L$
  are both given by the vanishing of the determinant,
and  the 3rd secant variety is the ambient space, hence $(X,L)$ is   a $brpp$.
On the other hand $x^3y\in S^4\BC^2$ has rank $4$ (see for instance \cite[Thm~23]{BGI}),
   but the maximal rank of any point in $S^2\BC^3$ is three
   (because $S^2\BC^3$ is a space of quadrics in three variables, and quadrics are diagonalizable) .
\end{example}

\begin{proposition} Strassen's conjecture~\ref{conj_strassen_rpp} and its border rank
version hold for $X:= Seg(\pp 1\times \BP B\times \BP C)$.

That is, for $L:=\CC \otimes B' \otimes C' \oplus \CC \otimes B'' \otimes C'' \subset \CC^2 \otimes B \otimes C$,
the pair $(X,L)$ is both   a \rpp{} and a \brpp{}.
\end{proposition}

\begin{proof}
  In this case   $X\cap L = \BP^0 \times \BP B'  \times \BP C'
     \sqcup \BP^0 \times \BP B''  \times \BP C''$.
  So any element in $L$ is of the form:
  \[
    p:=a_1 \otimes (b_1 \otimes c_1 + \dotsb + b_k \otimes c_k) +
       a_2 \otimes (b_{k+1} \otimes c_{k+1} + \dotsb + b_{k+l} \otimes c_{k+l}).
  \]
  Here $a_1, a_2$ is the basis of $\CC^2$ determined (up to scale) by splitting $\CC^2 = \CC \oplus \CC$,
    $\fromto{b_1}{b_k}$ are vectors in $B'$, $\fromto{b_{k+1}}{b_{k+l}}$ are vectors in $B''$
    and similarly for $\fromto{c_1}{c_k}$, $C'$, $\fromto{c_{k+1}}{c_{k+l}}$ and $C''$.
  We can assume that the $b_i$'s are linearly independent and the
$c_i$'s as well so that
  $\ur_{X \cap L}( p )=  R_{X \cap L}( p) = k+l$.
  After projection $\BP^1 \to \BP^0$ which maps both $a_1$ and $a_2$ to
a single generator of $\BC^1$,
  this element therefore becomes clearly of rank $k+l$.
  Hence both  $X$-rank and $X$-border  rank of $p$ are at least $k+l$.
\end{proof}

\begin{example}[Cases where \brpp{} version of Comon's conjecture holds]
If $\s_r(v_d(\pp n))$ is defined by flattenings, or more generally
by equations inherited from the tensor product space, such
as the Aronhold invariant (which is a symmetrized version
of Strassen's equations) then   the pair as in Conjecture~\ref{conj_comon_rpp} will be a $brpp_r$.
Set-theoretic defining equations for $\s_r(v_d(\pp n))$
are known for $d \ge 2r-1$ and either $n\le 3$ or $r \le 10$, see \cite[Thm~1.1]{BuBu}.
They are also known classically in the case $\s_r(v_d(\pp 1))$ for all $r,d$.
In all the known cases.
the equations are indeed inherited.

Regarding the rank version, it holds trivially  for general points
(as the \brpp{} version holds) and for points in
$\s_2(v_d(\pp n))$, as  a point not of honest rank two is
of the form $x^{d-1}y$, which gives rise to
$x\otc x\ot y + x\otc x\ot y\ot x+ \cdots + y\ot x\otc x$.
By \cite[Prop.~1.1]{BLsecant}
   one concludes.

If one would like to look for counter-examples, it might be
useful to look for linear spaces $M$ such that
$M\cap Seg(\pp n\ctimes \pp n)$ contains more than
$\tdim M+1$ points but
$L\cap M\cap Seg(\pp n\ctimes \pp n)$ contains the expected
number of points as these give rise to counter-examples to the \brpp{} version of Strassen's conjecture.
\end{example}

\bibliographystyle{amsplain}

\bibliography{BGLeis}

\end{document}